\documentclass{amsbook}
\usepackage{all,graphicx}

\newcommand{\bV}{\mathbf{V}}

\begin{document}

\begin{abstract} 

These lecture notes for the 2013 CIME/CIRM summer school
{\em Combinatorial Algebraic Geometry} 
deal with manifestly infinite-dimensional algebraic
varieties with large symmetry groups. So large, in fact, that subvarieties
stable under those symmetry groups are defined by finitely many orbits
of equations---whence the title {\em Noetherianity up to symmetry}. It
is not the purpose of these notes to give a systematic, exhaustive
treatment of such varieties, but rather to discuss a few ``personal
favourites'': exciting examples drawn from applications in algebraic
statistics and multilinear algebra. My hope is that these notes will
attract other mathematicians to this vibrant area at the crossroads
of combinatorics, commutative algebra, algebraic geometry, statistics,
and other applications. \\
\ \\
\noindent
{\bf Acknowledgments} The author was supported by a Vidi grant from the
Netherlands Organisation for Scientific Research (NWO).\\
\ \\
\hfill Eindhoven, Fall 2013\\ 
\hfill Jan Draisma

\end{abstract}

\title{Noetherianity up to symmetry}
\author{Jan Draisma}
\address{
Department of Mathematics and Computer Science, 
Eindhoven University of Technology; 
and Centrum Wiskunde en Informatica, Am\-ster\-dam; The Netherlands
\hfill {\tt j.draisma@tue.nl}
}

\date{Fall 2013}
\maketitle

\tableofcontents

\chapter{Kruskal's Tree Theorem} \label{ch:Kruskal}

All finiteness proofs in these lecture notes are based on a beautiful
combinatorial theorem due to Kruskal. In fact, the special case
of that theorem known as Higman's Lemma suffices for all of those
proofs. But, hoping that Kruskal's Tree Theorem will soon find further
applications in infinite-dimensional algebraic geometry, I have
decided to prove the theorem in its full strength. Original sources
for Kruskal's Tree Theorem and Higman's Lemma are \cite{Kruskal60}
and \cite{Higman52}, respectively. We follow closely the beautiful
proof in \cite{NashWilliams63}. Throughout we use the notation
$\NN:=\{1,2,\ldots\}$, $\ZZ_{\geq 0}:=\{0,1,\ldots\}$, and
$[n]:=\{1,\ldots,n\}$ for $n \in \ZZ_{\geq 0}$. In
particular, we have $[0]=\emptyset$. 

The main concept is that of a well-partial-order on a set $S$. This is
a partial order $\leq$ with the property that for any infinite sequence
$s_1, s_2, \ldots$ of elements of $S$ there exists a pair of indices
$i<j$ with $s_i \leq s_j$. Arguing by contradiction one then proves that
there exists an index $i$ such that $s_j \geq s_i$ holds for infinitely
many indices $j>i$. Take the first such index $i_1$, and retain only
the term $s_{i_1}$ together with the infinitely many terms $s_j$ with
$j>i_1$ and $s_j \geq s_{i_1}$. Among these pick an index $i_2>i_1$ 
in a similar fashion, etc. This leads to the conclusion that in a
well-partially-ordered set any infinite sequence has an infinite ascending
subsequence $s_{i_1} \leq s_{i_2} \leq \ldots$ with $i_1<i_2<\ldots$.

Examples of well-partial-orders are partial orders on finite sets,
and well-orders (which are linear well-partial-orders). If two
sets $S,T$ are both equipped with well-partial-orders, then the
componentwise partial order on the Cartesian product $S \times T$
defined by $(s,t) \leq (s',t')$ if and only if $s \leq s'$ and $t \leq
t'$ is again a well-partial-order. Indeed, in an infinite sequence
$(s_1,t_1),(s_2,t_2),\ldots$ there is an infinite subsequence of indices
where the $s_i$ increase weakly and in that subsequence there exist a
pair of indices $i \leq j$ where in addition to $s_i \leq s_j$ also the
inequality $t_i \leq t_j$ holds.

Repeatedly applying this Cartesian-product construction with all
factors equal to the non-negative integers $\ZZ_{\geq 0}$ one obtains
the statement that the componentwise order on $\ZZ_{\geq 0}^n$ is a
well-partial-order. This fact, known as Dickson's Lemma, can be used
to prove Hilbert's Basis Theorem. In a similar fashion we shall use
Kruskal's Tree Theorem to prove Noetherianity of certain rings up to
symmetry. Before stating and proving Kruskal's Tree Theorem, we first
discuss the following special case.

\begin{lm}
For any well-partial-order on a set $S$ the partial order on the set
of finite multi-subsets of $S$ defined by $A \leq B$ if and only if there
exists an injective map $f:A \to B$ with $a \leq f(a)$ for all $a \in A$
is a well-partial-order.
\end{lm}

\begin{proof}
Suppose that it is not. Then there exists an infinite sequence
$A_1,A_2,\ldots$ of finite multi-subsets of $S$ such that
$A_i \nleq A_j$ for all pairs of indices $i < j$. Such a sequence is called a
bad sequence, and we may assume that it is minimal in the following
sense. First, the cardinality $|A_1|$ of $A_1$ is minimal among all
bad sequences. Second, $|A_2|$ is minimal among all bad sequences
starting with $A_1$, etc. 

As the empty multi-set is smaller than all other
multi-sets, none of the multi-sets $A_i$ is
empty, so we may choose an element $a_i$ from each $A_i$ and define
$B_i:=A_i \setminus \{a_i\}$. There exists an infinite subsequence
$i_1<i_2<\ldots$ where $a_{i_1} \leq a_{i_2} \leq \ldots$. Now
the desired contradiction will follow by considering the sequence
\[ A_1,A_2,\ldots,A_{i_1-1},B_{i_1},B_{i_2},\ldots. \] 
Indeed, no $A_i$
with $i \leq i_1-1$ is less than or equal to $A_j$ with $i< j \leq
i_1-1$. But neither is any $A_i$ with $i \leq i_1-1$ less than or
equal to any $B_j$ with $j \in \{i_1,i_2,\ldots\}$, or else we would
have $A_i \leq B_j \leq A_j$, with the inclusion map witnessing the
second inequality. Finally, no relation $B_i \leq B_j$ holds with $i,j
\in \{i_1,i_2,\ldots\}$ and $i < j$. Indeed, a map $B_i \to B_j$
witnessing that inequality could be extended to a map $A_i \to A_j$
witnessing $A_i \leq A_j$ by mapping $a_i$ to $a_j$. We conclude that the
new sequence is bad, but this contradicts the minimality of $|A_{i_1}|$
among all bad sequences starting with $A_1,\ldots,A_{i_1-1}$.
\end{proof}

The general case of Kruskal's Tree Theorem concerns the set of
(isomorphism classes of) finite, rooted trees whose vertices are labelled
with elements of a fixed partially ordered set $S$. We call such objects
$S$-labelled trees. A partial order on $S$-labelled trees is defined
recursively as follows; see also Figure~\ref{fig:branching}. Suppose that
$T$ is an $S$-labelled tree with root $r$ and suppose that $T$ branches
at $r$ into trees $B_1,\ldots,B_{p}$ whose roots are the children of $r$;
$p=0$ is allowed here, and renders void one of the conditions that follow. We say that $T$ is less than or equal to a second $S$-labelled tree
$T'$ if the latter has a vertex $v$ (not necessarily its root) where $T'$
branches into trees $B'_1,\ldots,B'_{q}$ (rooted at children of $v$),
such that the $S$-label of $v$ is at least that of $r$ and such that
there exists an injective map $\pi$ from $[p]$ into $[q]$ with $B_i
\leq B'_{\pi(i)}$ for all $i$. Unfolding this recursive definition one
finds that the one-dimensional CW-complex given by the tree
$T$ can then be homeomorphically embedded into that of $T'$
in such a way that each vertex $u$ of $T$ gets mapped into a vertex of
$T'$ whose $S$-label is at least that of $u$; see
Figure~\ref{fig:homeo}.

\begin{figure}
\includegraphics[scale=.6]{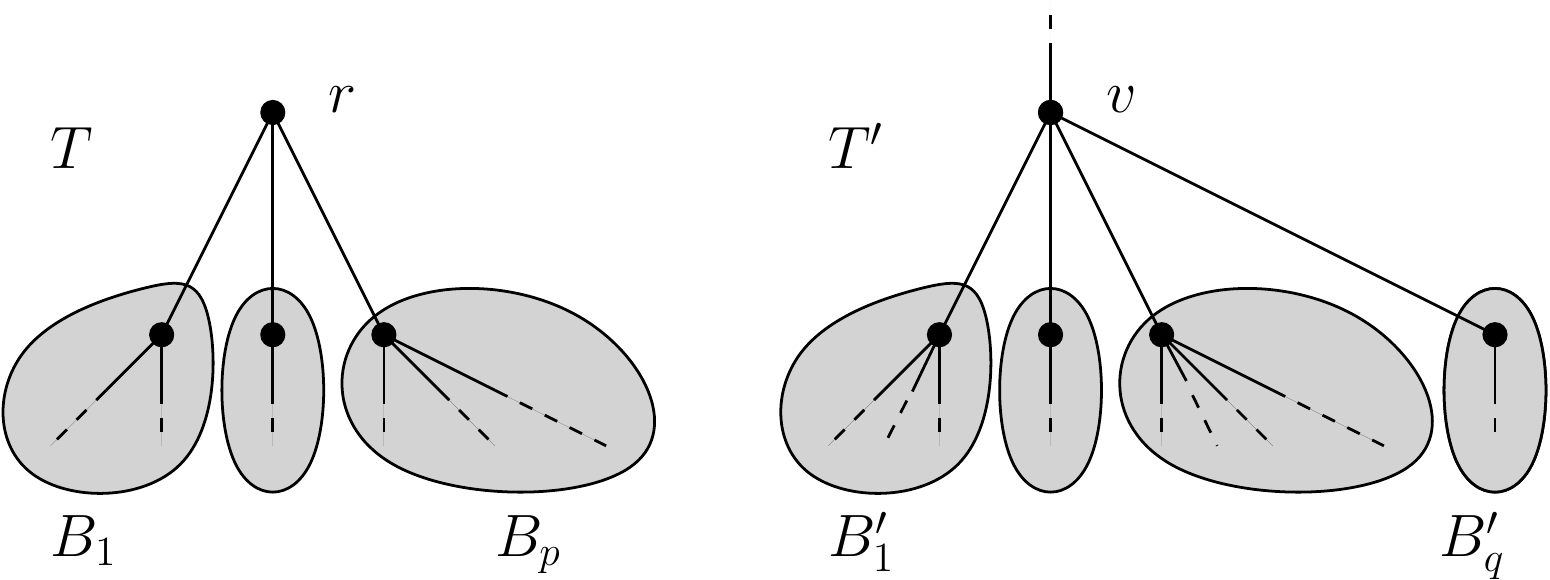}
\caption{If the $S$-label of $r$ is at most that of $v$, and if
$B_{i} \leq B'_{\pi(i)}$ for some injective $\pi:[p] \to [q]$, then $T
\leq T'$.}
\label{fig:branching}
\end{figure}

\begin{figure}
\includegraphics[scale=.6]{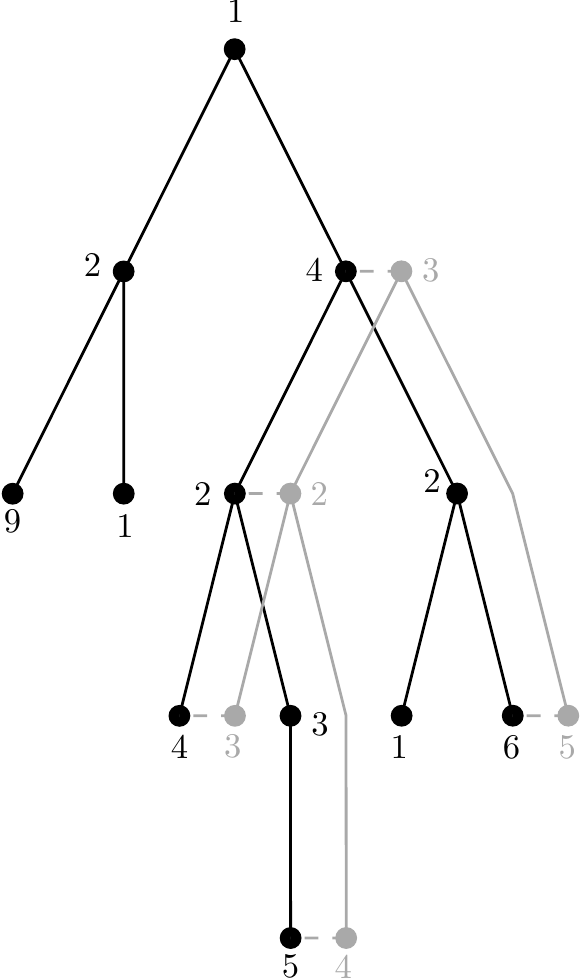}
\caption{The $\NN$-labelled tree in gray is smaller than the
$\NN$-labelled tree in black in Kruskal's order.}
\label{fig:homeo}
\end{figure}

\begin{thm}[Kruskal's Tree Theorem]
For any non-empty well-partially-ordered set $S$, the set of $S$-labelled
trees is well-partially-ordered by the partial order just defined.
\end{thm}

Note that the lemma is, indeed, a special case of this theorem, obtained
by restricting to star trees with a root (labelled with some fixed,
irrelevant element of $S$) connected directly to all its leaves.

\begin{proof}
The proof is very similar to that of the lemma. Assume, for a
contradiction, the existence of a bad sequence $T_1,T_2,\ldots$,
which we may take minimal in the sense that the cardinality of the
vertex set of $T_i$ is minimal among all bad sequences starting with
$T_1,\ldots,T_{i-1}$. At its root, $T_i$ branches into a finite multi-set
$R_i$ of smaller $S$-labelled trees. Let $R$ be the set-union of all $R_i$
as $i$ runs through $\NN$, so that $R$ is a set of
$S$-labelled trees. If $R$ contained a bad sequence, then
it would contain a bad sequence $B_{i_1},B_{i_2},\ldots$ with $B_{i_j}
\in R_{i_j}$ and $i_1<i_2<\ldots$. Then, as in the proof of the lemma,
one shows that the sequence $T_1,\ldots,T_{i_1-1},B_{i_1},B_{i_2},\ldots$
would be a bad sequence of trees contradicting the minimality of the
original sequence. Hence $R$ is well-partially-ordered.

Consider a subsequence $T_{k_1},T_{k_2},\ldots$ with $k_1<k_2<\ldots$
for which the root labels of the $T_{k_i}$ weakly increase in $S$.
Applying the lemma to the well-partially-ordered set $R$, we find that
there exist $j<l$ such that $R_{k_j} \leq R_{k_l}$. But then $T_{k_j}
\leq T_{k_l}$ (by mapping the root of $T_{k_j}$ to the root of $T_{k_l}$
and $R_{k_j}$ suitably into $R_{k_l}$), a contradiction.
\end{proof}

We now formulate Higman's Lemma, which is a useful direct consequence
of Kruskal's Tree Theorem. Given a partially ordered set $S$, define a
partial order on the set $S^*=\bigcup_p S^p$ of finite sequences over $S$
as follows. A sequence $(s_1,\ldots,s_{p})$ is less than or equal to a
sequence $(s'_1,\ldots,s'_{q})$ if there exists a strictly increasing
map $\pi:[p] \to [q]$ satisfying $s_i \leq s'_{\pi(i)}$ for all $i$.

\begin{thm}[Higman's Lemma]
For any well-partially-ordered set $S$, the partial order on $S^*$ just
defined is a well-partial-order.
\end{thm}

\begin{proof}
Encode a sequence $(s_1,\ldots,s_{p})$ in $S^*$ as an $S$-labelled tree with
root $1$ labelled $s_1$, with a single child $2$ labelled $s_2$, etc.
Under this encoding the partial order on $S^*$ agrees with that on
$S$-labelled trees, so that Kruskal's theorem implies Higman's Lemma.
\end{proof}

\chapter{Equivariant Gr\"obner bases} \label{ch:Groeb}

Just like, through a leading term argument, Dickson's Lemma implies
Hilbert's Basis Theorem, Higman's Lemma implies the central
finiteness result that all our later proofs build upon. What
follows is certainly not the most general setting, but it will
suffice our purposes.  For much more on this theme see
\cite{Cohen67,Cohen87,Aschenbrenner07,Aschenbrenner08,Drensky06,Hillar09,
Hillar13,LaScala09}.

Let $X$ be a (typically infinite) set of variables, and let $\Mon$ denote
the free commutative monoid of monomials in those variables. Let $\leq$
be a monomial order on $\Mon$, i.e., a well-order that satisfies the
additional condition that $u \leq v \Rightarrow uw \leq vw$ for all
$u,v,w \in \Mon$. Let $\Pi$ be a (typically non-commutative) monoid
acting from the left on $\Mon$ by means of monoid endomorphisms and
assume that the action preserves strict inequalities, i.e., $u < v$
implies $\pi u<\pi v$ for all $\pi \in \Pi$. In particular, $\pi$
acts by means of an injective map on $\Mon$. Moreover, we have $\pi u
\geq u$ since otherwise the sequence $u>\pi u>\pi^2 u>\ldots$ would
contradict that $\leq$ is a well-order.

Let $K$ be a field and denote by $K[X]=K\Mon$ the ring of polynomials
in the variables in $X$, or, equivalently, the monoid algebra over $K$ of
$\Mon$. The action of $\Pi$ on $\Mon$ extends to an action on $K[X]$
by means of ring endomorphisms preserving $1$. For a non-zero $f \in
K[X]$ denote by $\lmon(f) \in \Mon$ the largest monomial with a non-zero
coefficient in $f$. As the action of $\Pi$ preserves the (strict)
monomial order, we have $\lmon(\pi f)=\pi \lmon(f)$ in addition to the
usual properties of $\lmon$. In other words, the map $\lmon$ from $K[X]
\setminus \{0\}$ to $\Mon$ is $\Pi$-equivariant, and this motivates the
terminology in the following definition.

\begin{de}
Let $I$ be a $\Pi$-stable ideal in $K[X]$. Then a $\Pi$-Gr\"obner basis,
or equivariant Gr\"obner basis if $\Pi$ is clear from the context, of $I$
is a subset $B$ of $I$ with the property that for any $f \in I$ there
exists a $g \in B$ and a $\pi \in \Pi$ with $\lmon(\pi g) | \lmon(f)$.
\end{de}

The set $B$ is an equivariant Gr\"obner basis of $I$ if and only if the
union $\Pi B=\{\pi g \mid \pi \in \Pi, g \in B\}$ of the $\Pi$-orbits of
elements of $B$ is an ordinary Gr\"obner basis of $I$ (except that it will
typically not be finite). Then, in particular, $\Pi B$ generates $I$ as
an ideal; and we also say that $B$ generates $I$ as a $\Pi$-stable ideal.

We do not require that an equivariant Gr\"obner basis be finite, but
finite ones will of course be the most useful ones to us. To formulate
a criterion guaranteeing the existence of finite equivariant Gr\"obner
bases we define the $\Pi$-divisibility relation on $\Mon$ by $u |_\Pi v$
if and only if there exists a $\pi \in \Pi$ such that $\pi u$ divides $v$.
This relation is reflexive (take $\pi=1$), transitive (if $v=u' \cdot \pi
u$ and $w=v' \cdot \sigma v$, then $w=(v' \sigma u') \cdot (\sigma \pi)
u$), and antisymmetric (if $\pi u|v$ and $\sigma v|u$ then $u \leq \pi
u  \leq v$ and $v \leq \sigma v  \leq u$ so that $u=v$).

\begin{prop}[\cite{Hillar09}] \label{prop:GroebnerWellPartialOrder}
Every $\Pi$-stable ideal $I \subseteq K[X]$ has a finite $\Pi$-Gr\"obner basis
if and only if $|_\Pi$ is a well-partial-order.
\end{prop}

\begin{proof}
For the ``only if'' part observe that if $u_1,u_2,\ldots$ were a bad
sequence of monomials, then the $\Pi$-stable ideal generated by them,
i.e., the smallest $\Pi$-stable ideal containing them, would not have
a finite equivariant Gr\"obner basis. For the ``if'' part let $I$ be a
$\Pi$-stable ideal in $K[X]$. Let $M$ denote the set of $|_\Pi$-minimal
elements of $\{\lmon(f) \mid f \in I \setminus \{0\}\}$. As $|_\Pi$ is
a well-partial-order, $M$ is finite, say $M=\{u_1,\ldots,u_{p}\}$. Choose
$f_1,\ldots,f_{p} \in I\setminus\{0\}$ with $\lmon(f_i)=u_i$. Then
$\{f_1,\ldots,f_{p}\}$ is a $\Pi$-Gr\"obner basis of $I$.
\end{proof}

The main example that we shall use has $X:=\{x_{ij} \mid i \in [k], j
\in \NN\}$ and $\Pi:=\Inc(\NN)$, the monoid of maps $\NN \to \NN$ that
are strictly increasing in the standard order on $\NN$. This monoid acts
on $X$ by $\pi x_{ij}=x_{i\pi(j)}$; the action extends multiplicatively
to an action on $\Mon$ and linearly to an action by ring endomorphisms
on the polynomial ring $R:=K[X]=K[(x_{ij})_{ij}]$. There exist monomial
orders $\leq$ for which $u<v$ implies $\pi u < \pi v$; for instance,
the lexicographic order with $x_{ij} < x_{i'j'}$ if $i<i'$ or $i=i'$
and $j<j'$.

\begin{thm}[\cite{Cohen87,Hillar09}] \label{thm:kbyNmatrices}
Fix a natural number $k$. Then any $\Inc(\NN)$-stable ideal $I$
in the ring $K[x_{ij} \mid i \in [k], j \in \NN]$ has a finite
$\Inc(\NN)$-Gr\"obner basis with respect to any monomial order preserved
by $\Inc(\NN)$. In particular, any $\Inc(\NN)$-stable ideal $I$ in that
ring is generated, as an ideal, by finitely many $\Inc(\NN)$-orbits
of polynomials.
\end{thm}

\begin{proof}
By Proposition~\ref{prop:GroebnerWellPartialOrder} it suffices to prove
that $|_{\Inc(\NN)}$ is a well-partial-order. To this end, we shall apply
Higman's Lemma to $S=\ZZ_{\geq 0}^k$ with the componentwise partial order, which
is a well-partial-order by Dickson's Lemma. Encode a monomial $u$ in the
variables $x_{ij}$ as a word $(s_1,\ldots,s_{p})$ in $S^*$ as follows:
$p$ is the largest value of the column index $j$ for which some variable
$x_{ij}$ appears in $u$, and $(s_j)_i$ is the exponent of $x_{ij}$ in $u$.
Now given any sequence $u_1,u_2,\ldots$ of monomials, by Higman's
Lemma there exist indices $m<l$ such that the sequences $s,s'$ encoding
$u_m,u_l$ satisfy $s \leq s'$. This means that there exists a strictly
increasing map $\pi:[p] \to [p']$, with $p,p'$ the lengths of $s,s'$,
such that $s_j \leq s'_{\pi(j)}$ for all $j \in [p]$. Extend $\pi$ in
any manner to a strictly increasing map $\NN \to \NN$. Then the exponent
of any variable $x_{ij}$ in $\pi u_m$ equals $0$ if $j \not \in \pi([p])$
and $(s_{\pi^{-1}j})_i \leq (s'_j)_i$ otherwise. This proves that $\pi
u_m | u_l$, as desired.
\end{proof}

The second statement in Theorem~\ref{thm:kbyNmatrices} has several
consequences. One is that any ascending chain $I_1 \subseteq I_2
\subseteq \ldots$ of $\Inc(\NN)$-stable ideals in $R$ stabilises at
some finite index $n$: $I_n=I_{n+1}=\ldots$; we express this fact by
saying that $R$ is $\Inc(\NN)$-Noetherian. This implies that $R$ is
$\Syminf$-Noetherian, where the group $\Syminf:=\bigcup_{j \in \NN}
\Sym([j])$ is obtained by embedding $\Sym([j])$ into $\Sym([j+1])$ as
the stabiliser of $j+1$ and where $\pi \in \Syminf$ acts on $x_{ij}$ by
$\pi x_{ij}=x_{i\pi(j)}$.  Indeed, the $\Syminf$-orbit of any polynomial
$f$ contains the $\Inc(\NN)$-orbit of $f$, and hence any $\Syminf$-stable
ideal is also $\Inc(\NN)$-stable. Note that one can also replace
the countable group $\Syminf$ by the uncountable group of all permutations
of $\NN$, because the two have exactly the same orbits on $R$.

\begin{ex} \label{ex:NbyNmatrices}
In contrast to these beautiful positive results, consider the set
$X=\{y_{ij} \mid i,j \in \NN\}$ with $\Pi=\Inc(\NN)$-action given by
$\sigma y_{ij}=y_{\sigma(i)\sigma(j)}$. We claim that $|_\Pi$ is not a
well-partial-order. Indeed, consider the sequence of monomials
\[ y_{12}y_{21}, y_{12}y_{23}y_{31}, \ldots
\]
encoding directed cycles on two, three, etc.~vertices. Any $\pi \in
\Inc(\NN)$ maps such a monomial to a monomial representing another
directed cycle of the same length. Since no larger cycle contains a
smaller cycle as a subgraph, this is a bad sequence of monomials. The
same argument shows that $K[X]$ is not $\Syminf$-Noetherian. Similar
counterexamples exist for the action of $\Syminf \times \Syminf$ on $X$
given by $(\pi,\sigma) y_{ij}=y_{\pi(i) \sigma(j)}$.
\end{ex}

We now return to the general setting of equivariant Gr\"obner bases,
without the assumption that $|_\Pi$ is a well-partial-order.
These bases can sometimes be computed by a $\Pi$-equivariant version of
Buchberger's algorithm. The halting criterion in this equivariant
Buchberger algorithm is the following equivariant version of Buchberger's
criterion involving $S$-polynomials.

\begin{prop}[Equivariant Buchberger criterion]
Let $B$ be a subset of a $\Pi$-stable ideal $I$ in $K[X]$. Then $B$ is a
$\Pi$-Gr\"obner basis of $I$ if and only if for all $f,g \in B$ and all
$\sigma,\tau \in \Pi$ the ordinary $S$-polynomial $S(\sigma f,\tau g)$ gives
remainder $0$ upon division by $\Pi B$. 
\end{prop}

This criterion follows immediately from the ordinary Buchberger criterion
applied to $\Pi B$---indeed, while most textbooks assume a finite number
of variables, division-with-remainder and Buchberger's criterion apply
to infinitely many variables as well; the crucial ingredient is the
fact that the monomial order is a well-order. Unfortunately, since $\Pi$
is typically infinite, checking whether $B$ is an equivariant Gr\"obner
basis using the equivariant Buchberger criterion may be an infinite task,
even when $B$ is finite. But in many cases of interest this task can be
reduced to a finite task as follows.

Assume, first, that for any two polynomials $f,g \in K[X]$ the Cartesian
product $\Pi f \times \Pi g$ of the $\Pi$-orbits of $f$ and $g$ is the
union of finitely many diagonal orbits $\Pi (\sigma_i f, \tau_i g)=\{(\pi
\sigma_i f, \pi \tau_i g) \mid \pi \in \Pi\},\ i=1,\ldots,r$, where $r \in \NN$
and $\sigma_1,\tau_1,\ldots,\sigma_{r},\tau_{r} \in \Inc(\NN)$ are
allowed to depend on $f,g$. Then we would like to check only whether the
$S$-polynomials $S(\sigma_i f,\tau_i g)$ reduce to zero upon division by
$\Pi B$, and conclude that all $S(\sigma f,\tau g)$ reduce to zero. For
this we would like that $S(\pi \sigma_i f, \pi \tau_i g)=\pi S(\sigma_i
f,\tau_i g)$, because letting $\pi$ act on the reduction
of $S(\sigma_i f,\tau_i g)$ to zero yields a reduction of $S(\pi \sigma_i
f,\pi \tau_i g)$ to zero. This desired $\Pi$-equivariance of $S$-polynomials
does not follow from the assumptions so far, but it does follow if we
make the further assumption that each $\pi \in \Pi$ preserves least
common multiples, i.e., that $\lcm(\pi u,\pi v)=\pi \lcm(u,v)$ for all
$u,v \in \Mon$. This is, in particular, the case if $\pi$ maps variables
to variables.

\begin{thm}
Assume that Cartesian products $\Pi f \times \Pi g$ of $\Pi$-orbits on
$K[X]$ are unions of finitely many diagonal $\Pi$-orbits, and assume
that $\Pi$ preserves least common multiples of monomials. Let $S$
be a finite subset of $K[X]$ and consider the following algorithm:
\begin{enumerate}
\item Set $B:=S$ and $P:=\binom{S}{2} \cup \{(f,f) \mid f
\in S\}$, where $\binom{S}{2}$ is the set of pairs of
distinct elements from $S$;
\item If $P=\emptyset$, then stop, otherwise pick $(f,g) \in P$ and remove
it from $P$.
\item Choose $r \in \NN,\ \sigma_1,\tau_1,\ldots,\sigma_{r},\tau_{r} \in \Pi$ such that
$\Pi f \times \Pi g=\bigcup_{i=1}^{r} \Pi (\sigma_i f,\tau_i g)$. 
\item For each $i=1,\ldots,r$ do the following: 
reduce $S(\sigma_i f,\tau_i g)$ modulo
$\Pi B$, and if the remainder $h$ is non-zero, then add $h$ to $B$ and
consequently add $B \times \{h\}$ to $P$. 
\item Return to step 2.
\end{enumerate}
If and when this algorithm terminates, then $B$ is a $\Pi$-Gr\"obner
basis for the $\Pi$-stable ideal generated by $S$. Moreover, if $|_\Pi$ is
a well-partial-order, then this algorithm does terminate.
\end{thm}

As argued above, all but the last sentence of this theorem follows from
the ordinary Buchberger criterion. The last sentence follows as for the
ordinary Buchberger algorithm: if the algorithm does not terminate, then
an infinite number of non-zero remainders $h_1,h_2,\ldots$ are added. If
$|_\Pi$ is a well-partial-order, then there exist $i<j$ with $\lmon(h_i)
|_\Pi \lmon(h_j)$, which means that $h_j$ was not reduced with respect
to $h_i$, a contradiction.

One point to stress is that in initialising the pair set $P$ also pairs
of two identical polynomials $(f,f), f \in S$ need to be added; and that
similarly, when adding a remainder $h$ to $B$, also the pair $(h,h)$ needs
to be added to $P$. Indeed, already the $\Pi$-stable ideal generated by
a single polynomial can be interesting, as the following example shows.

\begin{ex} \label{ex:OneFactor}
Let $X=\{x_i \mid i \in \NN\} \cup \{y_{ij} \mid i,j \in \NN, i>j\}$
and let $\Pi=\Inc(\NN)$ act on $X$ by $\pi x_i=x_{\pi(i)}$ and $\pi
y_{ij}=y_{\pi(i)\pi(j)}$.  Set $S:=\{y_{21}-x_2x_1\}$ and let $I$ denote
the $\Pi$-stable ideal generated by $S$. We would like to compute an
equivariant Gr\"obner basis of the elimination ideal $I \cap K[y_{ij}
\mid i>j]$. To this end, we choose the lexicographic monomial order with
$x_1<x_2<\ldots$ and $y_{ij} < y_{kl}$ if either $i < k$ or $i=k$ and
$j < l$, and with $y_{kl}<x_i$ for all $i,k,l$. Note that $\Pi$ preserves the
strict monomial order and least common multiples. 

To apply the equivariant Buchberger algorithm we must further check that
Cartesian products of $\Pi$-orbits on $K[X]$ are finite unions of diagonal
$\Pi$-orbits. For this, let $f,g$ be elements of $K[X]$ and let $p,q$
be such that all variables in $f,g$ have indices contained in $[p],[q]$,
respectively. Then $\sigma f,\tau g$ depend only on the restrictions of
$\sigma$ and $\tau$ to $[p],[q]$ respectively. Enumerate all (finitely
many) pairs $(\sigma_i:[p] \to \NN,\tau_i:[q] \to \NN),\ i=1,\ldots,r$ of
increasing maps for which the union of $\im(\sigma_i)$ and $\im(\tau_i)$
equals some interval $[t]=\{1,\ldots,t\}$, necessarily with $t \leq
p+q$. Extend these $\sigma_i$ and $\tau_i$ arbitrarily to elements of
$\Pi$. We claim that for any pair $\sigma,\tau \in \Pi$ we have $(\sigma
f,\tau g)=(\pi\sigma_i f,\pi\tau_i g)$ for some $i \in [r]$ and some $\pi
\in \Pi$. Indeed, there exists a unique $i$ for which there exists an
(again, unique) increasing map $\pi:[t]=\im(\sigma_i) \cup \im(\tau_i) \to
\sigma([p]) \cup \tau([q])$ such that the restrictions of $\sigma,\tau$ to
$[p],[q]$ equal the restrictions of $\pi \circ \sigma_i,\pi \circ \tau_i$
to $[p],[q]$, respectively. Extend $\pi$ in any manner to an element of
$\Pi$ and we find $(\sigma f,\tau g)=(\pi \sigma_i f, \pi \tau_i g)$,
as desired.

This means that we can apply the equivariant Buchberger algorithm,
but without the guarantee that it terminates, since $|_\Pi$ is not a
well-partial-order (adapt Example~\ref{ex:NbyNmatrices} to see this). It turns out the
algorithm does terminate, though, and yields the following equivariant
Gr\"obner basis (after self-reduction): 
\begin{align*}
B=&
\{
x_1 x_2 - y_{21}, \\
&x_3 y_{21} - x_2 y_{31},\ 
x_3 y_{21} - x_1 y_{32},\  
x_2 y_{31} - x_1 y_{32},\\
&x_1^2 y_{32} - y_{31} y_{21}, \\
&y_{43} y_{21} - y_{41} y_{32},\  
y_{42} y_{31} - y_{41} y_{32}.
\}
\end{align*}
Since the monomial order is an elimination order, we conclude that
$I \cap K[y_{ij} \mid i>j]$ has a $\Pi$-Gr\"obner basis given by the
last two binomials. In particular, that ideal is generated, as an
$\Inc(\NN)$-stable ideal, by these binomials. The result
that we have just proved
by computer has first appeared as a theorem in \cite{deLoera95}. This
example gives the ideal of the so-called second hypersimplex, or, with
a slight modification, of the Gaussian one-factor model.  The $k$-factor
model for $k=2$ and higher will be the subject of Chapter
\ref{ch:KFactor}.
\end{ex}

\chapter{Equivariant Noetherianity}

In this chapter we establish a number of constructions of
equivariantly Noetherian rings and spaces. For some of the material
see \cite{Draisma08b}.

Given a ring $R$ (always commutative, with 1) and a monoid $\Pi$ with
a left action on $R$ by means of (always unital) endomorphisms we say
that $R$ is $\Pi$-Noetherian, or equivariantly Noetherian if $\Pi$
is clear from the context, if every chain $I_1 \subseteq I_2 \subseteq
\ldots$ of $\Pi$-stable ideals in $R$ eventually stabilises, that is,
if there exists an $n \in \NN$ for which $I_n=I_{n+1}=\ldots$. This is
equivalent to the condition that any $\Pi$-stable ideal $I$ is generated,
as an ideal, by finitely many $\Pi$-orbits $\Pi f_1,\ldots, \Pi f_s$. We
then say that $f_1,\ldots,f_s$ generate $I$ as a $\Pi$-stable ideal.

We have seen a major example in Chapter~\ref{ch:Groeb}, namely, for any
fixed natural number $k$ and any field $K$ the ring $K[x_{ij} \mid i
\in [k], j \in \NN]$ with its action of $\Inc(\NN)$ on the second index
is $\Inc(\NN)$-Noetherian.

There are several constructions of new equivariantly Noetherian
rings from existing ones. The first and most obvious is that if $R$
is $\Pi$-Noetherian and $I \subseteq R$ is a $\Pi$-stable ideal, then
$R/I$ is $\Pi$-Noetherian: any chain of $\Pi$-stable ideals in $R/I$
lifts to a chain of $\Pi$-stable ideals in $R$ containing $I$, and the
first chain stabilises exactly when the second chain does.

A second construction takes a $\Pi$-Noetherian ring $R$ to the polynomial
ring $R[x]$ in a variable $x$, where $\Pi$ acts only on the coefficients
from $R$. The standard proof of Hilbert's Basis Theorem, say from
\cite{Lang65}, generalises word by word from trivial $\Pi$ to general $\Pi$.

It is not true, in general, that a subring of an equivariantly Noetherian
ring is equivariantly Noetherian. Indeed, this is already not true for
ordinary Noetherianity, where $\Pi$ is the trivial monoid.  However,
the following construction proves that certain well-behaved subrings
of equivariantly Noetherian rings are again equivariantly Noetherian.
Suppose that $S$ is a subring of $R$ with the property that $R$ splits
as a direct sum $S \oplus M$ of $S$-modules. If $J$ is an ideal in $S$
and $I$ is the ideal in $R$ generated by $J$, then we claim that $S
\cap I=J$---indeed, any element $f$ of $S \cap I$ can be written as
$f=\sum_i f_i g_i$ with the $f_i$ elements of $J$ and the $g_i$ elements
of $R$. Applying the $S$-linear projection $\pi: R \to S$ along $M$
to both sides yields $f=\sum_i f_i \pi(g_i) \in J$, as claimed.  If,
moreover, the monoid $\Pi$ acts on $R$ and stabilises $S$, then $I$ is
$\Pi$-stable if $J$ is. We conclude that if $R$ is $\Pi$-Noetherian,
then any chain $I_1 \subseteq I_2 \subseteq \ldots$ of $\Pi$-stable 
ideals in $S$ generates such a chain $J_1 \subseteq J_2
\subseteq \ldots$
in $R$, and $J_n=J_{n+1}=\ldots$ implies that $I_n=J_n \cap S=J_{n+1}
\cap S=I_{n+1}=\ldots$; so $S$ is $\Pi$-Noetherian.

A particularly important example of this situation is the following
proposition, due to Kuttler. Suppose that a group $H$ acts on $R$ by
means of ring automorphisms, and that the action of $H$ commutes with
that of $\Pi$, i.e., for every $\pi \in \Pi$ and $h \in H$ and $f \in R$
we have $\pi h f = h \pi f$. Then the ring $R^H:=\{f \in R
\mid Hf=\{f\}\}$
of $H$-invariant elements of $R$ is stable under the action of $\Pi$.

\begin{prop} \label{prop:Reynolds}
If on the one hand $R$ is $\Pi$-Noetherian and on the other hand $R$
splits as a direct sum of irreducible $\ZZ H$-modules, then $R^H$ is also
$\Pi$-Noetherian.
\end{prop}

\begin{proof}
By the discussion preceding the proposition, we need only prove that $R$
splits as a direct sum $R^H \oplus M$ of $R^H$-modules. For this, split
$R$ as a direct sum $\oplus_i M_i$ of irreducible $\ZZ H$-modules $M_i$. Then
$R^H$ is the direct sum of the $M_i$ with trivial $H$-action, and we set
$M$ equal to the direct sum of the $M_i$ with non-trivial $H$-action.
We want to show that $f M_i \subseteq M$ for every $f \in R^H$ and every
$M_i \subseteq M$. To this end, let
$\rho:R \to R^H$ be the projection along $M$ and consider 
the map $M_i \to R^H$ sending
$m$ to $\rho(f m)$. By invariance of $f$ this map is $H$-equivariant,
and by irreducibility of $M_i$ its kernel is either $\{0\}$ or all of
$M_i$. But in the first case, the non-trivial $H$-module $M_i$ would be
embedded into the trivial $H$-module $R^H$, which is impossible. Hence
that kernel is all of $M_i$, and $fM_i \subseteq M$.
\end{proof}

In our applications, $R$ will typically be an algebra over some field
$K$, and $H$ will act $K$-linearly. Then it suffices that $R$ splits
as a direct sum of irreducible $KH$-modules (as one can infer from the
proposition by replacing $H$ by the group $K^* \times H$).

We give several applications of this proposition. First, we have seen in
Example~\ref{ex:NbyNmatrices} that the ring of polynomials in the entries
$y_{ij}$ of an $\NN \times \NN$-matrix is not $\Inc(\NN)$-Noetherian. The
following corollaries show that interesting quotients of such rings
are $\Inc(\NN)$-Noetherian.

\begin{cor} \label{cor:Segre}
Let $k$ be a natural number. Consider the homomorphism $\psi: K[y_\bm
\mid \bm \in \NN^k] \to K[x_{ij} \mid i \in [k], j \in \NN]$ sending
$y_\bm$ to $\prod_{i=1}^k x_{i,m_i}$. The kernel of $\psi$ is generated
by finitely many $\Inc(\NN)$-orbits of polynomials, and the quotient
$K[(y_\bm)_\bm]/\ker \psi$ is $\Inc(\NN)$-Noetherian.
\end{cor}

In more geometric language, that quotient is the coordinate ring of
$k$-dimensional infinite-by-infinite-by-\ldots-by-infinite tensors of
rank one.

\begin{proof}
The first statement follows from the standard fact that the ideal
of the variety of rank-one tensors is generated by the quadrics
$y_{m_0m_1}y_{m_0'm_1'}-y_{m_0m_1'}y_{m_0'm_1}$, where $m_0,m_0'$
are multi-indices of length equal to some $\ell \leq k$ and $m_1,m_1'$
are multi-indices of length $k-\ell$. The entries of the multi-indices
$m_0,m_0',m_1,m_1'$ taken together form a set of cardinality at most
$2k$, and this implies that each quadric of the form above is obtained
by applying some element of $\Inc(\NN)$ to one of the finitely many such
quadrics with all indices in the interval $[2k]$.

The second statement follows from the proposition (or rather the
discussion preceding it): the quotient is isomorphic, as a ring with
$\Inc(\NN)$-action, to the subring $S=\im \psi$ of $K[x_{ij} \mid i \in
[k], j \in \NN]$. The ring $S$ consists of all monomials in the $x_{ij}$
that involve equally many variables, counted with their exponents, from
all of the $k$ rows of the $k \times \NN$-matrix $(x_{ij})_{ij}$. If one
writes $M$ for the vector space complement of $S$ spanned by all other
monomials, then $M$ is an $S$-module, and the fact that $K[(x_{ij})_{ij}]$
is $\Inc(\NN)$-Noetherian implies that $S$ is.
\end{proof}

Alternatively, if $K$ is infinite, then one can characterise $S$ as the
set of $H$-invariants, where $H$ is the subgroup of $(K^*)^k$ consisting
of $k$-tuples with product $1$ and where $h$ acts on $x_{ij}$ by $h
x_{ij}=h_i x_{ij}$. Each monomial outside $S$ spans an irreducible,
non-trivial $H$-module, and the proposition implies that $S$ is
$\Inc(\NN)$-Noetherian.

A substantial generalisation of Corollary~\ref{cor:Segre}, which applies
to a wide class of monomial maps into $K[(x_{ij})_{ij} \mid i \in
[k], j \in \NN]$, is proved in \cite{Draisma13a}. For stabilisation of
appropriate lattice ideals, see \cite{Hillar13}.

The next corollary shows that determinantal quotients of the coordinate
ring of infinite-by-infinite matrices are $\Inc(\NN)$-Noetherian,
provided that the field has characteristic zero.

\begin{cor}
For any natural number $k$ and any field $K$ of characteristic zero,
the quotient of the ring $K[y_{ij} \mid i,j \in \NN]$ by the ideal $I_k$
generated by all $(k+1) \times (k+1)$-minors of the matrix
$(y_{ij})_{ij}$ is $\Inc(\NN)$-Noetherian.
\end{cor}

Note that the set of these determinants is the union of finitely many
$\Inc(\NN)$-orbits of equations, so that the corollary implies that any
$\Inc(\NN)$-stable ideal containing $I_k$ is generated by finitely many
$\Inc(\NN)$-orbits.

\begin{proof}
Let the group $H=\lieg{GL}_k$ act on the ring $K[x_{il} \mid i \in \NN, l
\in [k]]$ by $h x_{il}:=(xh)_{il}$, where $xh$ is the product of the $\NN
\times k$-matrix $x$ with variable entries and the $k \times k$-matrix
$h$. Similarly, let $H$ act on the ring $K[z_{lj} \mid l \in [k], j
\in \NN]$ by $h z_{lj}:=(h^{-1} z)_{lj}$. Note that both actions commute
with the action of $\Inc(\NN)$ on the indices $i,j$, respectively. Let $R$
be the polynomial ring $K[x_{il},z_{lj} \mid i,j \in \NN, l \in [k]]$,
equipped with the natural $\Inc(\NN)$-action and $H$-action. Classical
invariant theory tells us that rings, like $R$, on which $H$ acts as an
algebraic group split into a direct sum of irreducible $KH$-modules; here
we use that $\cha K=0$. So we may apply Proposition~\ref{prop:Reynolds}.

The First Fundamental Theorem \cite[Theorem 5.2.1]{Goodman09} for $H$
states that the algebra $R^H$ of $H$-invariant elements of the ring $R$ is
generated by all pairings $p_{ij}:=\sum_l x_{il} z_{lj}=(xz)_{ij}$. The
Second Fundamental Theorem \cite[Theorem 12.2.12]{Goodman09} states that
the kernel of the homomorphism $K[(y_{ij})_{ij}] \mapsto R$ determined by
$y_{ij} \mapsto p_{ij}$ is precisely $I_k$.  Thus the quotient by $I_k$
is isomorphic, as a $K$-algebra with $\Inc(\NN)$-action, to $R^H$. This
proves the corollary.

\end{proof}

Similar results are obtained by taking other rings with group actions
where the invariants and the polynomial relations among them are known.
Here is an example, which first appeared in \cite{Draisma08b}.

\begin{cor}
For any natural number $k$ and any field $K$ of characteristic zero,
the kernel of the homomorphism $\psi:K[y_\bm \mid \bm \in \NN^k,
m_1<\ldots<m_k] \to K[x_{ij} \mid i \in [k], j \in \NN]$ sending $y_\bm$
to the determinant of $x[\bm]$, the $k \times k$-submatrix of $x$ obtained
by taking the columns indexed by $\bm$, is generated by finitely many
$\Inc(\NN)$-orbits; and the quotient of $K[(y_{\bm})_{\bm}]$ by $\ker
\psi$ is $\Inc(\NN)$-Noetherian.
\end{cor}

\begin{proof}
Let $H=\lieg{SL}_k$, the group of $k \times k$-matrices of determinant
$1$, act on $K[(x_{ij})_{ij}]$ by $h x_{ij}=(h^{-1} x)_{ij}$. The First
Fundamental Theorem for $\lieg{SL}_n$ says that the $k\times k$-minors of $x$
generate the invariant ring of $H$, and the Second Fundamental Theorem
says that the Pl\"ucker relations among those determinants, which can
be covered by finitely many $\Inc(\NN)$-orbits, generate the ideal of
all relations. Now proceed as in the previous case.
\end{proof}

We remark that Alexei Krasilnikov showed that the Noetherianity of this
corollary does not hold when $\cha K=2$ and $k=2$. However,
a weaker form of Noetherianity, which we introduce now, does hold. 

Let $X$ be a topological space equipped with a right action of a
monoid $\Pi$ by means of continuous maps $X \to X$. Then we call $X$
equivariantly Noetherian, or $\Pi$-Noetherian, if every chain $X=X_0
\supseteq X_1 \supseteq X_2 \supseteq \ldots$ of closed, $\Pi$-stable
subsets stabilises.  If $R$ is a $K$-algebra with a left action of $\Pi$
by means of $K$-algebra endomorphisms, then for any $K$-algebra $A$ the
set $X:=R(A):=\Hom_K(R,A)$ of $A$-valued points of $R$ is a topological
space with respect to the Zariski topology in which closed sets are
defined by the vanishing of elements of $R$. Moreover, the monoid
$\Pi$ acts from the right on $R(A)$ by $(p \pi)(r)=p(\pi r)$. If $R$
is $\Pi$-Noetherian, then $R(A)$ is $\Pi$-Noetherian in the topological
sense. Conversely, if $R(A)$ is $\Pi$-Noetherian in the topological
sense for every $K$-algebra $A$, then $R$ is $\Pi$-Noetherian---indeed,
just take $A$ equal to $R$, so that the map that takes closed sets to
vanishing ideals is a bijection. However, topological Noetherianity of,
say, $R(K)$ does not necessarily imply Noetherianity of $R$. An example
of this phenomenon is given by Krasilnikov's example: the ring $K[\det
x[i,j] \mid i,j \in \NN, i<j]$, where $x$ is a $[2] \times \NN$-matrix
of variables, is not $\Inc(\NN)$-Noetherian if $\cha K=2$, but its set of
$K$-valued points is---indeed, this set of points is the image of $K^{[2]
\times \NN}$ under the $\Inc(\NN)$-equivariant map sending a matrix to
the vector of its $2 \times 2$-determinants. Since $K^{[2] \times \NN}$
is $\Inc(\NN)$-Noetherian, so is its image.

More generally, $\Pi$-equivariant images of $\Pi$-Noetherian topological
spaces are $\Pi$-Noetherian, and so are $\Pi$-stable subsets with the
induced topology. Another construction that we shall make much use of is
the following. 

\begin{prop} \label{prop:ZG}
Let $G$ be a group with a right action by homeomorphisms on a topological
space $X$. Let $\Pi$ be a submonoid of $G$ and let $Z$ be a $\Pi$-stable
subset of $X$. Assume that $Z$ is $\Pi$-Noetherian with the induced
topology. Then $Y:=ZG=\bigcup_{g \in G} Zg \subseteq X$ is $G$-Noetherian
with the induced topology.
\end{prop}

\begin{proof}
Let $Y=Y_1 \supseteq Y_2 \supseteq Y_3 \supseteq \ldots$ be a chain
of $G$-stable closed subsets of $Y$. Then each $Z_i:=Y_i \cap Z$ is
$\Pi$-stable and closed, hence by $\Pi$-Noetherianity of $Z$ there
exists an $n$ with $Z_n = Z_{n+1} = \ldots$. By definition of $Y$, for
each $y \in Y_i$ there exist a $g \in G$ and a $z \in Z$ with $y=zg$,
and by $G$-stability of $Y_i$ we have $z=yg^{-1} \in Z_i$.
This means that $Y_i$ can be recovered from $Z_i$ as
$Y_i=Z_iG $, and hence the chain $Y_1 \supseteq Y_2 \supseteq
Y_3 \supseteq \ldots$ stabilises at $Y_n$, as well.
\end{proof}

\chapter{Chains of varieties} \label{ch:Chains}

In the remainder of these notes we study various chains of interesting
embedded finite-dimensional varieties, for which we want to prove that
from some member of the chain on, all equations for later members
come from those of earlier members by applying symmetry. To use the
infinite-dimensional techniques from the previous chapters, we first pass
to a projective limit, prove that the limit is defined by finitely many
orbits of equations, and from this fact we derive the desired result
concerning the finite-dimensional varieties. In this short chapter we
set up the required framework for this, again without trying to be as
general as possible. Most of this material is from \cite{Draisma08b}.

Thus let $K$ be a field and let $R_1,R_2,\ldots$ be commutative
$K$-algebras with $1$. The algebra $R_i$ plays the role of coordinate ring
of the ambient space of the $i$-th variety in our chain. Assume that the
$R_i$ are linked by (unital) ring homomorphisms $\iota_i:R_i \to R_{i+1}$
and $\pi_i:R_{i+1} \to R_i$ satisfying $\pi_i \circ \iota_i=1_{R_i}$.
Then we can form the $K$-algebra $R_\infty:=\bigcup_{i \in \NN} R_i$
with respect to the inclusions $\iota_i$; the use of the $\pi_i$ will
become clear later.

Suppose, next, that are given ideals $I_i \subseteq R_i$ such that
$\pi_i$ maps $I_{i+1}$ into $I_i$ and $\iota_i$ maps $I_i$ into
$I_{i+1}$. The ideal $I_i$ plays the role of defining ideal of the
$i$-th variety in our chain. Writing $S_i:=R_i/I_i$ we find that the
$\iota_i,\pi_i$ induce inclusions $S_i \to S_{i+1}$ and surjections
$S_{i+1} \to S_i$, respectively, and we set $I_\infty:=\bigcup_i I_i$
and $S_\infty:=\bigcup_i S_i$, which also equals $R_\infty / I_\infty$.

Assume, next, that a group $G_i$ acts on $R_i$ from the left by means
of $K$-algebra automorphisms, and that we are given embeddings $G_i \to G_{i+1}$
that render both $\iota_i$ and $\pi_i$ equivariant with respect to
$G_i$. Suppose furthermore that each $I_i$ is $G_i$-stable, which
expresses that the $i$-th variety has the same symmetries as imposed
on the ambient space. We form the group $G_\infty:=\bigcup_i G_i$, which
acts on $R_\infty,I_\infty,S_\infty$ by means of automorphisms.

For any $K$-algebra $A$, we write $R_i(A), S_i(A), R_\infty(A),
S_\infty(A)$ for the sets of $A$-valued points of these algebras,
i.e., for the set of homomorphisms $R_i \to A$, etc. As customary in
algebraic geometry, for a $p$ in these point sets, we write $f(p)$
rather than $p(f)$ for the evaluation of $p$ on an element $f$ in the
corresponding algebra. These sets are topological spaces with respect
to the Zariski topology, in which closed sets are of the form $\{p \in
R_i(A) \mid J(p)=\{0\}\}$ for some ideal $J$ in $R_i$, and similarly
for the other algebras. On these topological spaces $G_i$ or $G_\infty$
acts by means of homeomorphisms. Our set-up so far is summarised in the
diagram of Figure~\ref{fig:Chains}, where $\iota^*,\pi^*$ are the
pull-backs of $\iota$ and $\pi$, respectively.

\begin{figure}
\includegraphics{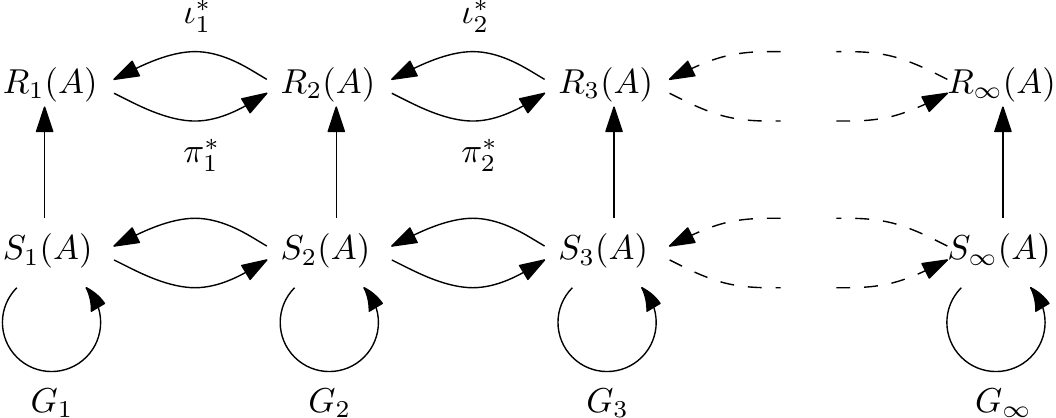}
\caption{Chains of varieties}
\label{fig:Chains}
\end{figure}

The relation $\iota^* \circ \pi^*=1$ implies that $\iota^*$ is surjective
(and not just dominant) and that $\pi^*$ is injective; indeed, the latter
is a closed embedding.  Still, $\pi^*$ is needed only a bit later.

The topological space $R_\infty(A)$ is canonically the same
as the projective limit, in the category of topological spaces, of the
spaces $R_i(A)$ with their Zariski topologies: First, at the level of
sets, an $A$-valued point $p$ of $R_\infty$ gives rise, by composition
with the embeddings $R_i \to R_\infty$, to homomorphisms $p_i:=R_i
\to A$ for all $i \in \NN$. The resulting sequence $(p_1,p_2,\ldots)$
has the property that the pull-back of $\iota_i^*$ maps $p_{i+1}$ to
$p_i$, i.e., it is a point of the inverse limit $\lim_{\ot i} R_i(A)$
of sets. Conversely, a point of this inverse limit gives homomorphisms
$p_i:R_i \to A$ such that $p_{i+1} \circ \iota_i=p_i$, and together these
define a homomorphism $R_\infty \to A$.  Second, the projective limit
topology on $R_\infty(A)$ is the weakest topology that renders all maps
$R_\infty(A) \to R_i(A)$ continuous. This means, in particular, that sets
given by the vanishing of a single element of $R_i \subseteq R_\infty$
must be closed, and so must intersections of these, which are sets
given by the vanishing of an ideal in $R_\infty$.  This shows that the
projective limit topology on $R_\infty(A)$ has at least as many closed
sets as the Zariski topology, and the converse is also clear since the
maps $R_\infty(A) \to R_i(A)$ are continuous in the Zariski topology.

The basic result that we shall use throughout the rest of the notes is
the following, where the use of the $\pi_i$ becomes apparent.

\begin{prop} \label{prop:Chains}
Let $i_0 \in \NN$ and assume that the set $S_\infty(A)$
is characterised inside $R_\infty(A)$ by the vanishing of all $g f \in
R_\infty$ with $g \in G_\infty$ and $f \in I_{i_0}$. Then for $i \geq i_0$
the set $S_i(A)$ is characterised by the vanishing of all functions of
the form $\pi_i \cdots \pi_{l-1} g f$ with $l \geq i$, $g \in G_l$,
and $f \in I_{i_0}$.
\end{prop}

\begin{proof}
That these functions vanish on $S_i(A)$ follows from the inclusion
$I_{i_0} \subseteq I_l$, the fact that $G_l$ stabilises $I_l$, and
the fact that $\pi_j$ maps $I_{j+1}$ into $I_j$. Conversely, suppose
that $p_i \in R_i(A)$ is a zero of all functions in the proposition,
and let $p=(p_1,p_2,\ldots)$ be the point of $R_\infty(A)$ obtained
by setting $p_{j+1}:=\pi_j^* p_j$ for $j \geq i$ and $p_j:=\iota_j^*
p_{j+1}$ for $j<i$. Then $p$ is a point in $S_\infty(A)$ by the assumed
characterisation of the latter set: any $g \in G_\infty$ lies in $G_l$
for some $l$ which we may take larger than $i$, and for $f \in I_{i_0}$
we have
\[ (g f)(p)=(g f)(p_l)=(g f)(\pi_{l-1}^* \cdots \pi_i^* p_i)
=(\pi_{l-1} \cdots \pi_i g f)(p_i)=0,\] 
as desired. In particular, this means that $p_i$ lies in $S_i(A)$. 
\end{proof}

It would be more elegant to characterise $S_i(A)$ for $i \geq i_0$ as the
common vanishing set of $G_i I_{i_0}$, i.e., of all functions of the form
$g f$ with $f \in I_{i_0}$ and $g \in G_i$. For this we introduce an
additional condition. For $l \geq i$ write $\iota_{il}:R_i \to R_l$
for the composition $\iota_{l-1} \cdots \iota_i$ and $\pi_{li}:R_l \to
R_i$ for the composition $\pi_i \cdots \pi_{l-1}$. The condition that
we want is:
\begin{quotation} 
For all indices $l,i_0,i_1$ with $l \geq i_0,i_1$ and for all $g \in G_l$
there exist an index $j \leq i_0,i_1$ and group elements $g_0 \in G_{i_0}, g_1 \in G_{i_1}$ such that 
\begin{equation} \label{eq:Giotapi}
\tag{*}
\pi_{li_1} \ g \ \iota_{i_0 l} = g_1 \  \iota_{ji_1} \  \pi_{i_0 j} \ g_0
\end{equation}
holds as an equality of homomorphisms $R_{i_0} \to R_{i_1}$.
\end{quotation}
This guarantees that the functions $\pi_{li} g f=\pi_{li} g \iota_{i_0 l}
f$ from the proposition can be written as $g_1 \iota_{ji} \pi_{i_0 j} g_0
f$ for some $j \leq i,i_0$ and $g_0 \in G_{i_0}$ and $g_1 \in G_i$. Since
$I_{i_0}$ is $G_{i_0}$-stable and $\pi_{i_0 j}$ maps $I_{i_0}$ into $I_j
\subseteq I_{i_0}$ the latter expression is an element of $G_i I_{i_0}$.

The discussion so far concerned an arbitrary, fixed $K$-algebra
$A$. In several applications we shall just take $A$ equal to $K$,
and the conclusion is that the point sets $S_i(A)$ for $i \geq i_0$ are
defined set-theoretically by equations coming from $I_{i_0}$ using
symmetry. However, if one assumes that $G_\infty I_{i_0}$ generates
$I_\infty$, then the assumption in the proposition holds for all
$K$-algebras $A$, hence so does the conclusion. From this one can conclude
that for $i \geq i_0$ the functions featuring in the proposition generate
the ideal $I_i$. Under the additional assumption \eqref{eq:Giotapi}
one finds that $G_i I_{i_0}$ generates the ideal $I_i$. We conclude this
chapter with a well-known example which paves the way for the treatment
of the $k$-factor model in the next chapter.

\begin{ex} \label{ex:Rankk}
Fix a natural number $k$. For $n \in \NN$ let $R_n$ be the polynomial
ring over $K$ in the $\binom{n+1}{2}$ variables $y_{ij}=y_{ji}$ with $i,j
\leq n$. Let $\iota_n$ be the natural inclusion $R_n \to R_{n+1}$ and
let $\pi_n$ be the projection $R_{n+1} \to R_n$ mapping all variables
to zero that have one or both indices equal to $n+1$. Then we have
$\pi_n \iota_n=1$ as required. Let $I_n \subseteq R_n$ be the ideal of
polynomials vanishing on all symmetric $n \times n$-matrices over $K$
of rank at most $k$. Then $\iota_n$ maps $I_n$ into $I_{n+1}$ since the
upper-left $n \times n$-block of an $(n+1)\times(n+1)$-matrix of rank
at most $k$ has itself rank at most $k$, and $\pi_n$ maps $I_{n+1}$
into $I_n$ since appending a zero last row and column to any matrix
yields a matrix of the same rank. Let $G_n:=\Sym(n)$ act on $R_n$ by $g
y_{ij}=y_{g(i),g(j)}$, and embed $G_n$ into $G_{n+1}$ as the stabiliser of
$n+1$. Then $G_n$ stabilises $I_n$ and the maps $\pi_n$ and $\iota_n$ are
$G_n$-equivariant. So we are in the situation discussed in this chapter.

Even the additional assumption \eqref{eq:Giotapi} holds. Indeed, consider
the effect of appending $l-i_1$ zero rows and columns to a symmetric
$i_1 \times i_1$-matrix, then simultaneously permuting rows and columns,
and finally forgetting the last $l-i_0$ rows and columns of which,
say, $m$ come from the zero rows and columns introduced in the first
step. Set $j:=i_1-(l-i_0-m)$, which is the number of rows and columns
of the original matrix surviving this operation. Then the same effect
is obtained by first permuting (with $g_1$) rows and columns such that
the $i_1-j$ rows and columns to be forgotten are in the last $i_1-j$
positions, then forgetting those rows and columns, then appending
$l-i_1-m=i_0-j$ zero rows and columns, and finally suitably permuting
rows and columns with a $g_0 \in \Sym(i_0)$.

It is known, of course, that if $K$ is infinite, then $I_n$ is generated
by all $(k+1) \times (k+1)$-minors of the matrix $(y_{ij})_{ij}$. This
implies that $I_\infty$ is generated by $G_\infty I_{2k+2}$; here 
$2k+2$ is the smallest size where representatives of all 
$\Sym(\NN)$-orbits of minors of an infinite symmetric matrix can be seen. 
But conversely, if through some other method (computational
or otherwise) one can prove that $I_\infty$ is indeed generated by
$G_\infty I_{2k+2}$, then by the discussion above this implies that $I_n$
is generated by $G_n I_{2k+2}$ for all $n \geq 2k+2$. Using
the equivariant Gr\"obner basis techniques from Chapter~\ref{ch:Groeb}
one can prove such a statement automatically for small $k$, say $k=1$
and $k=2$, much like we have done in Example~\ref{ex:OneFactor}.
\end{ex}

\chapter{The independent set theorem}

To appreciate the results of this chapter---though not to understand
the proofs---one needs some familiarity with Markov bases and their
relation to toric ideals. We formulate and prove the independent set
theorem from \cite{Hillar09}, first conjectured in \cite{Hosten07},
directly at the level of ideals.

Fix a natural number $m$, let $\Gamma$ be a subset of
$2^{[m]}$ (thought of as a hypergraph with vertex set $[m]$;
see Figure~\ref{fig:Gamma} for an example), and fix an 
infinite field $K$; what follows will, in fact, not depend on $K$.
To any $m$-tuple $r=(r_1,\ldots,r_m) \in \NN^{[m]}$ of natural numbers we
associate the polynomial ring 
\[ R_r:=K[y_{i_1,\ldots,i_{m}} \mid (i_1,\ldots,i_m) \in
[r_1] \times \cdots \times [r_t]] \]
in $\prod_{t \in [m]} r_t$-many variables, and the polynomial ring 
\[ Q_r:=K[x_{F,(i_t)_{t \in F}} \mid F \in 
\Gamma \text{ and } \forall t \in F: i_t \in [r_t]]. \]
Furthermore, we define 
$I_r \subseteq R_r$ as the kernel of the homomorphism $R_r \to Q_r$
mapping $y_{(i_t)_{t \in [m]}}$ to $\prod_{F \in \Gamma} x_{F,(i_t)_{t
\in F}}$. On $R_r,Q_r$ acts the group $G_r:=\prod_t \Sym([r_t])$ by
permutations of the variables, and the homomorphism defining $I_r$
is $G_r$-equivariant.  We want to let some of the $r_t$, namely,
those with $t$ in a given subset $T \subseteq [m]$, tend to infinity,
and conclude that the ideals $I_r$ stabilise up to $G_r$-symmetry. To
put this statement in the context of Chapter~\ref{ch:Chains}, we give the
$r_t$ with $t \not \in T$ some fixed values, and take the $r_t, t \in T$
all equal, say to $n \in \NN$. The corresponding rings and ideals are
called $R_n, Q_n, I_n$. We have inclusions $\iota_n: R_n \to R_{n+1}$
obtained by inclusion of variables and projections $\pi_n: R_{n+1}
\to R_n$ obtained by mapping all variables to zero that have at least
one $T$-labelled index equal to $n+1$. This maps $I_{n+1}$ into $I_n$
because there is a compatible homomorphism $Q_{n+1} \to Q_n$ setting
the relevant variables equal to zero. As in Chapter~\ref{ch:Chains}
we write $R_\infty,I_\infty$ for the union of all $R_n,I_n$.

The group $\Sym(n)^{T}$ acts on $R_n,Q_n,I_n$, but in fact the
independent set theorem only needs one copy of $\Sym(n)$ acting
diagonally. The additional assumption~\eqref{eq:Giotapi} holds by the
same reasoning as in Example~\ref{ex:Rankk}.

\begin{thm}[Independent set theorem \cite{Hillar09}.]
Suppose that $T \subseteq [m]$ is an independent set in $\Gamma$, i.e.,
that $T$ intersects any $F \in \Gamma$ in at most one element. Then
there exists an $n_0$ such that $I_n$ is generated by $\Sym(n) I_{n_0}$
for all $n \geq n_0$.
\end{thm}

The condition that $T$ is an independent set cannot simply be dropped. For
instance, if $m=3$ and $\Gamma=\{\{1,2\},\{2,3\},\{3,1\}\}$ (the model
of {\em no three-way interaction}), then if $r_1=n$ tends to infinity
for fixed $r_2,r_3$, then the ideal stabilises (see \cite{Aoki03} for the
case of $r_2=r_3=3$); but if $r_1=r_2=n$ both tend to infinity and $r_1$
is fixed, say, to $2$, then the ideal does not stabilise \cite{Diaconis98}.

\begin{figure}
\begin{center}
\includegraphics{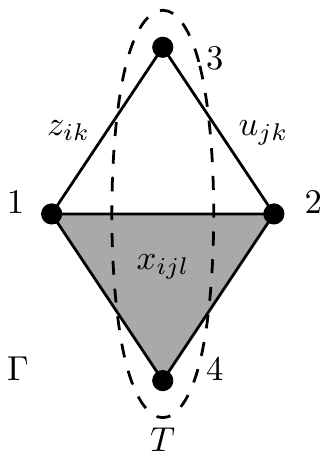}
\caption{A hypergraph $\Gamma$ with parameters
$x_{ijl},z_{ik},u_{jk}$ and independent set $T$.}
\label{fig:Gamma}
\end{center}
\end{figure}

\begin{ex}
As an example take $m=4$ and $\Gamma=\{124,13,23\}$, where
$124$ is short-hand for $\{1,2,4\}$, etc. We write 
\[ y_{ijkl}, x_{ijl},
z_{ik}, u_{jk} \text{ instead of }y_{i_1,i_2,i_3,i_4}, x_{124,(i_1,i_2,i_4)},
x_{13,(i_1,i_3)}, x_{23,(i_2,i_3)}, \]
respectively; see Figure~\ref{fig:Gamma}. The ideal
$I_n$ is the kernel of the homomorphism sending $y_{ijkl}$ to
$x_{ijl}z_{ik}u_{jk}$. Take $T=\{3,4\}$, an independent set in
$\Gamma$. The ideal $I_\infty$ contains obvious quadratic binomials
such as
\[
y_{ijkl}y_{ijk'l'}-y_{ijkl'}y_{ijk'l}
\]
with $i \in [r_1], j \in [r_2], k,l,k',l' \in \NN$. Indeed, the first monomial
maps to $x_{ijl}z_{ik}u_{jk}\cdot x_{ijl'}z_{ik'}u_{jk'}$, and the second
monomial maps to $x_{ijl'}z_{ik}u_{jk}\cdot x_{ijl}z_{ik'}u_{jk'}$, which is
the same thing.
\end{ex}

These obvious binomials generalise verbatim to the general case, where
they read
\begin{equation}  \label{eq:Binomial}
y_{jkl}y_{jk'l'} - y_{jkl'} y_{jk'l} 
\end{equation}
where now $j$ runs over the finite set $\prod_{t \in [m] \setminus T}
[r_t]$, $k,k'$ run over $\NN^S$ and $l,l'$ run over $\NN^{T \setminus
S}$ for some $S$ that runs over all subsets of $T$. Indeed, for any
variable $x=x_{F,(i_t)_{t \in F}}$ at most one $t \in F$ lies in
$T$. If such a $t$ exists for $x$ and lies in $S$, then whether $x$ appears
in the image of a variable $y_{jkl}$ does not depend on the value of
$l$. But disregarding that third index, the two monomials above are
the same. A similar reasoning for the case where $t \in T \setminus S$
and for the case where $F \cap T = \emptyset$ shows that $x$ has
the same exponent in the image of both monomials in the binomial above.

By Proposition~\ref{prop:Chains} relating chains to infinite-dimensional
varieties, we are done if we can prove that $I_\infty$ is generated by
$\Sym(\infty) I_n$ for some finite $n$. Let $J_n, J_\infty \subseteq
I_n, I_\infty$, respectively, be the ideals generated by all quadratic
binomials as in \eqref{eq:Binomial}. We claim that $J_\infty$ is
generated by $\Sym(\infty)J_{2|T|}$. Indeed, for any binomial as above,
the set of all indices appearing in $k,l,k',l'$ has cardinality
$n \leq 2|T|$, and there exists a bijection in $\Sym(\infty)$
mapping its support bijectively onto $[n]$, witnessing that the
binomial lies in $\Sym(\infty)J_{2|T|}$. The remainder of the proof
consists of showing that the quotient $R_\infty/J_\infty$ is, in fact,
$\Sym(\infty)$-Noetherian. To this end, we introduce a new
polynomial ring 
\[ P:=K\left[ y'_{jtq} \mid j \in \prod_{t \in [m] \setminus T}
[r_t],\ t \in T,\text{ and } q \in \NN \right], \] 
where the $y'_{jtq}$ are new variables, 
and consider the subring $R'_\infty$ of $P$ 
generated by all monomials $m_{ji}:=\prod_{t \in T}
y'_{jti_t}$ with $j$ as before and $i \in \NN^T$. The monomials $m_{ji}$
satisfy the binomial relations~\eqref{eq:Binomial} (for all splittings of $i$
into two subsequences $k$ and $l$), and it is known that these binomials
generate the ideal of relations among the $m_{ji}$---indeed, the ring $R'_\infty$ is the
coordinate ring of the Cartesian product of $\prod_{t \in [m] \setminus
T} r_T$-many copies of the variety of pure $|T|$-dimensional tensors. Thus
we have an isomorphism $R'_\infty \cong R_\infty/J_\infty$, and we want
to show that $R'_\infty$ is $\Sym(\infty)$-Noetherian. The enveloping
polynomial ring $P \supseteq R'_\infty$ is $\Sym(\infty)$-Noetherian
by Theorem~\ref{thm:kbyNmatrices} (only the index $q$ of the variables 
$y'_{jtq}$ is unbounded), but passing to a subring one may, in
general, lose Noetherianity. However, let $H$ be the torus in $(K^*)^T$
consisting of $T$-tuples of non-zero scalars whose product is $1$. Then
the $m_{ji}$ are $H$-invariant, and these monomials generate the ring
of $H$-invariant polynomials (if, as we may assume, $K$ is
infinite). Hence by Proposition~\ref{prop:Reynolds} we may conclude that
$R'_\infty$ is $\Sym(\infty)$-Noetherian, and this concludes the proof
of the independent set theorem.

\chapter{The Gaussian $k$-factor model} \label{ch:KFactor}

This chapter discusses finiteness results for a model from algebraic
statistics known as the Gaussian $k$-factor model.  General stabilisation
results for this model were first conjectured in \cite{Drton07}, and
for $1$ factor established prior to that in \cite{deLoera95}. For $2$
factors, a positive-definite variant was established in \cite{Drton08},
and an ideal-theoretic variant in \cite{Brouwer09e}. The ideal-theoretic
version for more factors is open, but the set-theoretic version was
established in \cite{Draisma08b}.

The Gaussian $k$-factor model consists of all covariance matrices for a
large number $n$ of jointly Gaussian random variables consistent with the
hypothesis that those variables can be written as linear combinations
of a small number $k$ of hidden factors plus independent, individual
noise. Algebraically, let $R_{n}$ be the $K$-algebra of polynomials
in variables $y_{ij}=y_{ji}$ with $i,j \in [n]$, and let $P_{kn}$ be
the $K$-algebra of polynomials in the variables $x_{il},\ i \in [n],
l \in [k]$ and further variables $z_1,\ldots,z_n$. Let $I_{kn}$ be the
kernel of the homomorphism $\phi_{kn}: R_n \to P_{kn}$ that maps $y_{ij}$
to the $(i,j)$-entry of the matrix
\[ x \cdot x^T + \diag(z_1,\ldots,z_n), \]
where we interpret $x$ as an $[n] \times [k]$-matrix of variables. Set
$S_{kn}:=R_{n}/I_{kn}$; the set $S_{kn}(K) \subseteq K^{\binom{n}{2}}$
of $K$-valued points of $S_{kn}$ is (the Zariski closure of) the Gaussian
$k$-factor model. Observe that $\phi_{kn}$ is $\Sym([n])$-equivariant,
so that $I_{kn}$ is $\Sym([n])$-stable. We are in the setting of
Chapter~\ref{ch:Chains}, with the map $\pi_n:R_{n+1} \to R_n$ mapping
$y_{i,(n+1)}$ equal to $0$ for all $i$ and the map $\iota_n:R_n \to
R_{n+1}$ the inclusion.  The technical assumption~\eqref{eq:Giotapi}
from that chapter holds for the same reason as in Example~\ref{ex:Rankk}.

\begin{thm}
For every fixed $k \in \NN$, there exists an $n_k \in \NN$ such that
for all $n \geq n_k$ the variety $S_{kn}(K) \subseteq K^{\binom{n+1}{2}}$
is cut out set-theoretically by the polynomials in $\Sym(n) I_{n_k}$.
\end{thm}

\begin{proof}
By Proposition~\ref{prop:Chains} and the discussion following its
proof we need only prove that $S_{k\infty}(K)$ is the zero set of
$\Sym(\infty) I_{n_k}$, for suitable $n_k$. Let $J_{k\infty}$ denote
the ideal generated by all $(k+1) \times (k+1)$-minors of the
symmetric $\NN \times \NN$-matrix $y$ that do not involve diagonal
entries of $y$, and set $S'_{k \infty}:=R_{\infty} / J_{k\infty}$. Then
surely $J_{k\infty}$ is contained in $I_{k\infty}$, so that, dually,
$S'_{k \infty}(K)$ contains $S_{k \infty}(K)$.

We claim that $S'_{k \infty}(K)$ is a $\Sym(\infty)$-Noetherian
topological space, and to prove this claim we proceed by induction. For
$k=0$ the equations in $J_{k\infty}=J_{0\infty}$ force all off-diagonal
entries of the matrix $y$ to be zero, so that $S'_{0 \infty}(K)$ is just
the set of $K$-points of $K[y_{11},y_{22},\ldots]$, with $\Sym(\infty)$
permuting the coordinates. The latter ring is $\Sym(\infty)$-Noetherian
by Theorem~\ref{thm:kbyNmatrices}, and hence its topological space of
$K$-points is certainly $\Sym(\infty)$-Noetherian.

\begin{figure}
\begin{center}
\includegraphics{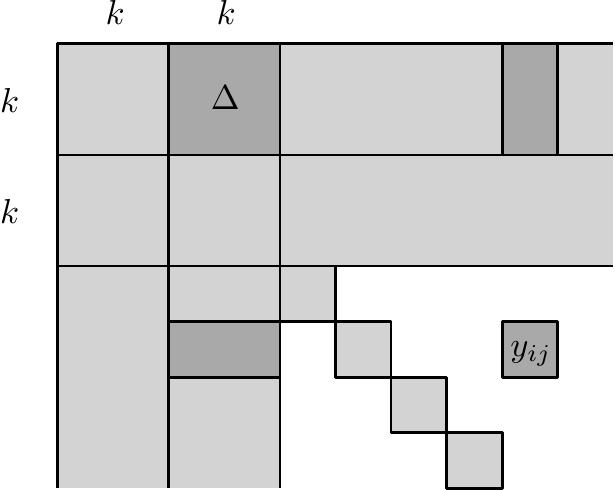}
\caption{The set $T$ labelling symmetric matrix entries (in light gray), and
the $(k+1)\times(k+1)$-determinant expressing $y_{ij}$ as a rational
function in the $T$-labelled variables (in dark gray).}
\end{center}
\end{figure}

Next, assume that $S'_{k-1,\infty}(K)$ is $\Sym(\infty)$-Noetherian,
and note that $S'_{k-1,\infty}(K)$ is a (closed) subset of
$S'_{k,\infty}(K)$. On any point outside this closed subset at least
one of the $k \times k$-determinants in $J_{k-1,\infty}$ is non-zero. Up
to signs, these determinants form a single orbit under
$\Sym(\infty)$, so if we set $\Delta:=\det
y[[k],[2k]\setminus[k]]$, then we have 
\begin{align*} S'_{k,\infty}(K)&=S'_{k-1,\infty}(K) \cup Z\Sym(\infty),
\text{ where}\\
Z&=\{y \in S'_{k \infty}(K) \mid \Delta \neq 0 \}.
\end{align*}
The union of two $\Sym(\infty)$-Noetherian topological spaces is
$\Sym(\infty)$-Noetherian, so it suffices to prove that $Z\Sym(\infty)$
is $\Sym(\infty)$-Noetherian. Observe that $Z$ itself is stable under the
subgroup $H:=\{\pi \in \Sym(\infty) \mid \pi|_{[2k]}=1|_{[2k]}\}$. Hence
by Proposition~\ref{prop:ZG} it suffices to prove that $Z$ is
$H$-Noetherian. To this end, define $T \subseteq \NN \times \NN$ by
\[ T:=\{(i,j) \in \NN \times \NN \mid i=j \text{ or } (i<j \text{ and }
i \in [2k])\}. \]
Let $Q$ be the open subset of $K^T$ where the $[k]
\times ([2k]\setminus [k])$-submatrix has non-zero determinant. The
coordinate ring of $Q$ is $H$-Noetherian by
Theorem~\ref{thm:kbyNmatrices} and
the fact that adding finitely many $H$-fixed variables preserves
$H$-Noetherianity. As a consequence, $Q$ is an $H$-Noetherian space. We
claim that the projection $\pr: Z \to Q$ that maps a matrix $y$ to
its $T$-labelled entries is a closed, $H$-equivariant embedding.
Equivariance is immediate. To see that $\pr$ is injective observe that,
for $y \in Z$, any matrix entry $y_{ij}$ with $i < j$ (since we work
with symmetric matrices) and $i \not \in [2k]$ satisfies an equation
\[0=\det y\left[[k] \cup \{i\},([2k] \setminus [k])\cup \{j\}\right] 
= \Delta \cdot y_{ij} - E, \]
where $E$ is an expression involving only variables in $T$. Since
$\Delta$ is non-zero, we find that $y_{ij}$ is
determined by $\pr(y)$. This shows injectivity. That $\pr:Z \to Q$
is, in fact, a closed embedding follows by showing that the dual map $K[Q]
\to K[Z]$ is surjective: the regular function $\frac{E}{\det \Delta}$ maps
onto $y_{ij}|_Z$. Since $Q$ is $H$-Noetherian, so is $Z$, and
as mentioned before this concludes the induction step. 

As $S'_{k,\infty}(K)$ is Noetherian, we find that in particular, the
Zariski closure $S_{k,\infty}(K)$ is cut out from $S'_{k,\infty}(K)$
by finitely many $\Sym(\NN)$-orbits of equations. Representatives of
these orbits already lie in $S'_{k,n_k}(K)$ for suitable $n_k$.
\end{proof}

\chapter{Tensors and $\Delta$-varieties} 
\renewcommand{\Seg}{\mathbf{Seg}}

This chapter deals with finiteness results for a wide class of varieties
of tensors, introduced by Snowden \cite{Snowden10} and called {\em
$\Delta$-varieties}.  The proof of this chapter's theorem is more
involved than earlier proofs, and we have therefore decided to break 
the chapter up into more digestible sections.

\subsection*{$\Delta$-varieties} 
We work over a ground field $K$, which we assume to be infinite
to avoid anomalies with the Zariski topology. 
For any tuple
$(V_1,\ldots,V_n)$ of finite-dimensional vector spaces over $K$,
we write $\bV(V_1,\ldots,V_{n})$ for the tensor product $V_1^*
\otimes \cdots \otimes V_{n}^*$. These spaces have three types of
interesting maps between them. First, given linear maps $f_i:V_i
\to W_i$ there is a natural linear map $\bV(W_1,\ldots,W_{n}) \to
\bV(V_1,\ldots,V_{n})$, namely, the tensor product $\otimes_i f_i^*$ of
the dual maps $f_i^*$. Second, given any $\sigma \in \Sym([n])$, there
is a canonical map $\sigma:\bV(V_{\sigma(1)},\ldots,V_{\sigma(n)})\to
\bV(V_1,\ldots,V_{n})$. Third, there is a canonical {\em flattening}
map $\bV(V_1,\ldots,V_{n},V_{n+1}) \to \bV(V_1,\ldots,V_{n} \otimes
V_{n+1})$, which is called like this because, in coordinates, it takes
an $(n+1)$-way table of numbers and transforms it into an $n$-way table;
see Figure~\ref{fig:flattening}.

\begin{figure}
\begin{center}
\includegraphics{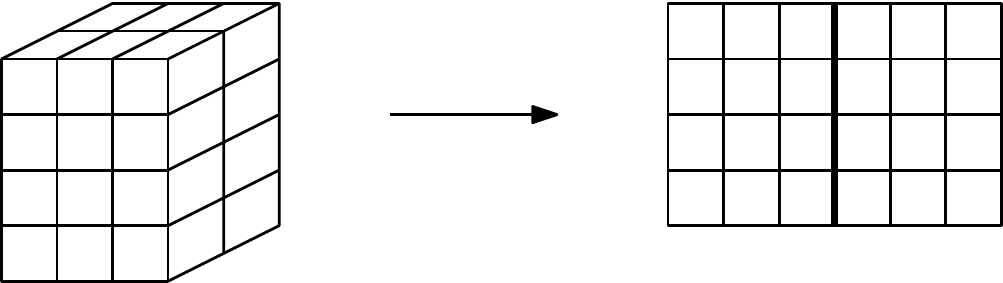}
\caption{Flattening an element of $K^4 \otimes K^2 \otimes K^3$ to an
element of $K^4 \otimes (K^2 \otimes K^3)$.}
\label{fig:flattening}
\end{center}
\end{figure}

A $\Delta$-variety is not a single variety, but rather a rule
$\bX$ that takes as input a finite sequence $(V_j)_{j \in [n]}$
of finite-dimensional vector spaces over $K$, and assigns to it a
subvariety $\bX(V_1,\ldots,V_{n})$ of $\bV(V_1,\ldots,V_{n})$. To
be a $\Delta$-variety, each of the three types of maps
above must preserve $\bX$, i.e., $\otimes_i f_i^*$ must map
$\bX(W_1,\ldots,W_{n})$ into $\bX(V_1,\ldots,V_{n})$, and
$\sigma$ must map $\bX(V_{\sigma(1)},\ldots,V_{\sigma(n)})$
into $\bX(V_1,\ldots,V_{n})$, and the flattening map must map
$\bX(V_1,\ldots,V_{n+1})$ into $\bX(V_1,\ldots,V_{n} \otimes V_{n+1})$.

The $\Delta$-varieties that we shall study will have a fourth,
additional property, namely, that the inverse to the isomorphism
$\bV(V_1,\ldots,V_{n},K) \to \bV(V_1,\ldots,V_{n} \otimes
K)=\bV(V_1,\ldots,V_{n})$ maps $\bX(V_1,\ldots,V_{n})$ into
$\bX(V_1,\ldots,V_{n},K)$; we call such $\Delta$-varieties {\em good}.
Taking any linear function $f$ from an additional vector space $V_{n+1}$
to $K$ we then find that $\bX(V_1,\ldots,V_{n}) \otimes f$,
being the image of $\bX(V_1,\ldots,V_{n})$ under the map
above followed by
$1_{V_1}^* \otimes \cdots \otimes 1_{V_{n}}^* \otimes f^*$, is
contained in $\bX(V_1,\ldots,V_{n},V_{n+1})$.
A (boring) example of a $\Delta$-variety that is not good is
that for which $\bX(V_1,\ldots,V_{n})$ equals $\bV(V_1,\ldots,V_{n})$
if $n<10$ and the empty set otherwise.

A typical example of a good $\Delta$-variety is $\Seg$, the cone over
Segre, which maps a tuple of vector spaces to the variety of pure tensors
$v_1 \otimes \cdots \otimes v_{n}$ in the tensor product of the duals. In
fact, any non-empty good $\Delta$-variety contains $\Seg$---but the class
of good $\Delta$-varieties is much larger. For instance, it is closed
under taking joins: if $\bX$ and $\bY$ are $\Delta$-varieties, then the
rule $\bX+\bY$ that assigns to $(V_i)_{i \in [n]}$ the Zariski closure of
$\bX(V_1,\ldots,V_{n})+\bY(V_1,\ldots,V_{n})$ is also a $\Delta$-variety,
and good if both $\bX$ and $\bY$ are. Similarly, (good) $\Delta$-varieties
are closed under taking tangential varieties, unions, and intersections.

Given some equations for an instance of a $\Delta$-variety $\bX$,
one obtains equations for other instances of $\bX$ by pulling
back along sequences of the maps appearing in the definition
of a $\Delta$-variety. For instance, start with the $2 \times
2$-determinant defining $\Seg(K^2,K^2)$ inside $\bV(K^2,K^2)$. Then
we obtain generators for the ideal of $\Seg(K^m,K^n)$ by pulling
the determinant back along duals of linear maps $f_1:K^2 \to K^m$ and
$f_2:K^2 \to K^n$.\footnote{Snowden chose the notion of $\Delta$-varieties
contravariant in the linear maps $f_i$ so as to make defining ideals and
more general {\em $\Delta$-modules} \cite{Snowden10} depend covariantly on
them.} And then, using the remaining two axioms, we also find equations
for the variety of pure tensors in, say, $\bV(K^2,K^2,K^2,K^2)$
through the flattening maps into $\bV(K^2 \otimes K^2,K^2 \otimes
K^2) \cong \bV(K^4,K^4)$, $\bV(K^2, K^2 \otimes K^2 \otimes K^2) \cong
\bV(K^2,K^8)$, etc. Indeed, one readily shows that one obtains generators
of the ideals of all instances of $\Seg$ in this manner.  The result in
this chapter is that a similar result holds for any sufficiently small
good $\Delta$-variety, at least at a topological level.

\begin{thm} \label{thm:Delta}
Let $\bX$ be a good $\Delta$-variety which is {\em bounded} in the
sense that there exist finite-dimensional vector spaces $W_1,W_2$ such
that $\bX(W_1,W_2)$ is not all of $\bV(W_1,W_2)$. Then there exist an
$n_\bX \in \NN$ and vector spaces $U_1,\ldots,U_{n_\bX}$ such that
$\bX$ equals the inclusion-wise largest $\Delta$-variety $\bY$ with
$\bY(U_1,\ldots,U_{n_\bX})=\bX(U_1,\ldots,U_{n_\bX})$.
\end{thm}

This means, in more concrete terms, that the equations for
$\bX(U_1,\ldots,U_{n_\bX})$, pulled back along all four types of linear
maps from the definition of a good $\Delta$-variety, yield equations
that cut out all instances of $\bX$ from their ambient spaces. In
particular, there is a universal degree bound, depending only on $\bX$
but not on $n$ or $V_1,\ldots,V_{n}$, on equations needed to define
$\bX(V_1,\ldots,V_{n})$ set-theoretically within
$\bV(V_1,\ldots,V_{n})$. 

Since $\GL(W_1) \times \GL(W_2)$ acts with a dense orbit on
$\bV(W_1,W_2)$---namely, the two-tensors (or matrices) of full rank---the
boundedness condition on $X$ implies that all two-tensors in instances of
the form $\bX(V_1,V_2)$ have uniformly bounded rank. This readily implies
that the boundedness condition on $\bX$ is also preserved under joins (by
adding the rank bounds), tangential varieties, intersections, and unions,
so that the theorem applies to a wide class of $\Delta$-varieties of
interest in applications. 

\subsection*{Related literature}
The boundedness condition on $\Delta$-varieties was first formulated, at
an ideal-theoretic level, in \cite{Snowden10}.  There it is conjectured
that a generalisation of Theorem~\ref{thm:Delta} should hold, for
bounded $\Delta$-varieties, on the ideal-theoretic level; and not only
for equations of instances of $\bX$, but also for their $q$-syzygies
for any fixed $q \geq 1$. This general statement is proved for $\Seg$
in \cite{Snowden10}. The special case where $q=1$, i.e., finiteness
of equations, is known to hold for the tangential variety to $\Seg$
\cite{Oeding11}, confirming a conjecture from \cite{Landsberg07};
and for the variety $\Seg+\Seg=2\Seg$ of tensors of border rank
at most $2$ \cite{Raicu10}, confirming the GSS-conjecture from
\cite{Garcia05} (a set-theoretic version of which was first proved
in \cite{Landsberg04}). The set-theoretic theorem above was first
proved in \cite{Draisma11d} for $k\Seg$, i.e., for any fixed secant
variety of $\Seg$; and a discussion with Snowden led to the insight
that our proof generalises to bounded, good, $\Delta$-varieties as
in the theorem. Further recent keywords closely related to the topic
of this chapter are $\GL_\infty$-algebras \cite{Sam12},
twisted commutative algebras \cite{Sam12b}, FI-modules
\cite{Church12}, and cactus varieties \cite{Buczynska13}.

\subsection*{From a $\Delta$-variety to an infinite-dimensional variety}

We prove the theorem by embedding all relevant instances of $\bX$ into a
single, infinite-dimensional variety given by determinantal equations,
and showing that this variety is Noetherian up to symmetry preserving
$\bX$. By the boundedness assumption, we can choose a number $p$ that
is strictly greater than the ranks of all two-tensors in instances
$\bX(V_1,V_2)$, independently of $V_1$ and $V_2$. Set $V:=K^{[p]}$ and
$X_n:=\bX(V,\ldots,V) \subseteq \bV_n:=\bV(V,\ldots,V)$, where the
number of $V$s equals $n$. We first argue that the equations for all
(infinitely many) $X_n, n \in \NN$ pull back to equations defining all
instances of $\bX$.

Indeed, let $V_1,\ldots,V_{n}$ be vector spaces, and let $\omega
\in \bV(V_1,\ldots,V_{n})$ be a tensor.  Then we claim that $\omega$
lies in $\bX(V_1,\ldots,V_{n})$ if and only if for all linear maps
$f_i:V \to V_i$ the image of $\omega$ under $\otimes_{i \in [n]} f_i^*$
lies in $X_n$. The ``only if'' claim follows from the first axiom for
$\Delta$-varieties. For the ``if'' claim, note that if $(\otimes_{i \in
[n]} f_i^*)\omega$ lies in $X_n$ for all tuples of $f_i$, then for each $j
\in [n]$, the linear map that $\omega$ induces from $\bigotimes_{i \neq j}
V_i$ into $V_j^*$, being a flattening of $\omega$ in $\bV(\bigotimes_{i
\neq j} V_i,V_j)$, has image $U_j \subseteq V_j^*$ of dimension strictly
smaller than $p$.  Now take $f_j:V \to V_j$ such that $f_j^*$ restricts
to an injection $U_j \to V^*$, and let $g_j:V_j \to V$ be such that
$g_j^* \circ f_j^*$ restricts to the identity on $U_j$.  Then the tensor
$\omega':=\otimes_{j} f_j^* \omega$ lies in $X_n=\bX(V,\ldots,V)$ by
assumption. But then, by the first axiom, the tensor $\otimes_j g_j^*
\omega'=\omega$ lies in $\bX(V_1,\ldots,V_{n})$, as claimed. This
argument actually also works ideal-theoretically; only later shall we
need to work purely topologically.

We now cast the chain of varieties $(X_n)_{n \in \NN}$ into the framework
of Chapter~\ref{ch:Chains}. To this end, let $R_n$ denote the symmetric
$K$-algebra generated by $V^{\otimes [n]}$, which is the coordinate ring
of $\bV_n$. Pick a non-zero element $x_0 \in V$
and let $\iota_n:R_n \to R_{n+1}$ be the homomorphism of $K$-algebras
determined by the linear map $V^{\otimes[n]} \to V^{\otimes
[n+1]},\ x
\mapsto x \otimes x_0$. The group $\GL(V)^{[n]}$ acts on $V^{\otimes
[n]}$ in the natural manner, and this extends to an action by
algebra automorphisms on $R_n$. Similarly, the group $\Sym([n])$
acts on $V^{\otimes [n]}$ by permuting tensor factors. The embedding
$\iota_n$ is equivariant for the group $G_n:=\Sym([n]) \ltimes
\GL(V)^{[n]}$ if we embed $\Sym([n])$ into $\Sym([n+1])$ by fixing 
$n$ and $\GL(V)^{[n]}$ into $\GL(V)^{[n+1]}$ by adding $1_V$ in the
last component.

The linear map $\iota^*: \bV_{n+1} \to \bV_n$ maps $X_{n+1}$ into $X_n$,
and $X_n$ is preserved by $G_n$. Letting $S_n$ be the coordinate ring of
$R_n$, we have all the arrows in the diagram of Figure~\ref{fig:Chains}
except for the arrows to the right. To obtain these, we use that $\bX$
is good, as follows. Given any $e_0 \in V^*$ such that $e_0(x_0) =1$,
the map $\pi_n^*:\bV_n \to \bV_{n+1}, \omega \mapsto \omega \otimes e_0$
maps $X_n$ into $X_{n+1}$. The dual to this linear map, extended to an
algebra homomorphism, is the required map $\pi:R_{n+1} \to R_n$. This
completes the diagram. The technical condition \eqref{eq:Giotapi} from
page~\pageref{eq:Giotapi} is also satisfied, i.e., for all indices
$l,i_0,i_1$ with $l \geq i_0,i_1$ and for all $g \in G_l$ there exist
an index $j \leq i_0,i_1$ and group elements $g_0 \in G_{i_0}, g_1 \in
G_{i_1}$ such that
\[ 
(\pi_{li_1} \ g \ \iota_{i_0 l})^* = (g_1 \  \iota_{ji_1} \  \pi_{i_0 j} \
g_0)^*.
\]
Indeed, the left-hand side is the composition of the map $\bV_{i_1} \to
\bV_{l}$ tensoring with $e_0^{\otimes l-i_1}$, followed by $g$, followed
by the map $\bV_l \to \bV_{i_0}$ contracting with $x_0^{\otimes l-i_0}$
in the last $l-i_0$ factors. Let $i_1-j$ be the number of factors $V^*$
in $\bV_{i_1}$ that are moved, by $g$, into the last $l-i_0$ positions
and hence end up being contracted in the last step. This means that
$j \leq i_0,i_1$ is the number of factors $V^*$ in $\bV_{i_1}$ that
are not contracted. Hence the composition can also be obtained by first
applying a $g_1 \in G_{i_1}$, ensuring that the $i_1-j$ factors $V^*$ in
$\bV_{i_1}$ that need to be contracted are in the last $i_1-j$ positions;
then contracting by a pure tensor in those positions, which for suitable
choice of $g_1$ may be chosen $x_0^{\otimes i_1-j}$; then tensoring with
$i_0-j$ copies of $e_0$; and finally applying a suitable element $g_0
\in G_{i_0}$.

The upshot of this is that if we can prove that the projective
limit $X_\infty:=\lim_{\ot n} X_n$ is defined by finitely
many $G_\infty:=\bigcup_n G_n$-orbits of equations within
$\bV_\infty:=\lim_{\ot n} \bV_n$, then there exists an $n_X$ such that
for all $n \geq n_X$ the variety $X_n$ is defined by the $G_n$-orbits
of equations for $X_{n_X}$. This implies the theorem.

\subsection*{Flattening varieties}

To prove this finiteness result, then, we show that $X_\infty$
is contained in a $G_\infty$-Noetherian subvariety $Y_\infty$ of
$\bV_\infty$, which we call a flattening variety, and that $Y_\infty$
itself is defined by finitely many $G_\infty$-orbits of equations. To
define $Y_\infty$, let $\bY^{(k)}$ denote the largest $\Delta$-variety
for which $\bY^{(k)}(V_1,V_2)$ consists of two-tensors of rank at
most $k$. Then $\bY^{(k)}(V_1,\ldots,V_n)$ is defined by the vanishing of
all $(k+1) \times (k+1)$-minors of the flattenings $\bV(V_1,\ldots,V_n)
\to \bV(\bigotimes_{i \in A}V_i,\bigotimes_{i \in B}V_i)$ for all
partitions of $[n]$ into disjoint subsets $A$ and $B$.
Set $Y^{(k)}_n:=\bY^{(k)}(V,\ldots,V) \subseteq
\bV_n$, and $Y^{(k)}_\infty:=\lim_{\ot n} Y^{(k)}_n \subseteq \lim_{\ot n}
\bV_n$. By the boundedness assumption on $\bX$, $Y^{(p-1)}_\infty$ contains
$X_\infty$.

We first prove that each $Y^{(k)}_\infty, k \in \NN$ is defined by
finitely many $G_\infty$-orbits of equations. Unwinding the definitions,
this statement boils down to the statement that if $\omega \in \bV_n$
with $n \gg 0$ does not lie in $Y^{(k)}_n$, then there exists an $i \in
[n]$ and an $x \in V$ such that contracting $\omega$ with $x$ in the
$i$-th position yields a tensor $\omega' \in \bV_{n-1}$ that does not
lie in $Y^{(k)}_{n-1}$. In fact, we shall see that $n>2k$ suffices. The
condition that $\omega$ does not lie in $Y^{(k)}_n$ means that there is
a partition $[n]=A \cup B$ such that $\omega$, regarded as a linear map
$V^{\otimes A} \to (V^*)^{\otimes B}$, has rank strictly larger than
$k$. Using that $n>2k$ and after swapping $A$ and $B$ if necessary we
may assume that $|B|>k$.

Let $U \subseteq (V^*)^{\otimes B}$ be a $(k+1)$-dimensional subspace
of the image of this linear map $\omega$. We claim that since $|B|$ is
larger than $k$, there exists a position $i \in B$ and an $x \in V$
such that the image of $U$ under contraction with $x$ in the $i$-th
position still has dimension $k+1$. Indeed, otherwise $U$ would be a
point in the projective variety
\begin{align*} Q:=\{W \in \Gr_{k+1}(V^*)^{\otimes B}
	 \mid &\text{ contracting $U$ with any $x$ in any
	position}\\ &\text{ decreases the dimension}\}. 
\end{align*}
We claim that this variety is empty. To prove this, extend the
distinguished vector $x_0 \in V$ from the definition of $\bV_\infty$
to a basis $x_0,\ldots,x_{p-1}$, where the distinguished $e_0
\in V^*$ vanishes on $x_1,\ldots,x_{p-1}$. Then the basis of
$V^*$ dual to $x_0,\ldots,x_{p-1}$ starts with $e_0$; denote
it $e_0,\ldots,e_{p-1}$.\footnote{The reason for labelling with
$\{0,\ldots,p-1\}$ rather than $[p]$ will become apparent soon.}

If $Q$ is not empty, then by Borel's fixed point theorem \cite{Borel91}
$Q$ contains a $T^B$-fixed point $W$, where $T$ is the maximal torus
in $\GL(V)$ consisting of invertible linear maps whose matrices with
respect to $x_0,\ldots,x_{p-1}$ are diagonal. This means that $W$ has a
basis of common eigenvectors $e_\alpha:=\otimes_{i \in B} e_{\alpha_i}$,
where $\alpha$ runs through some set $J \subseteq \{0,\ldots,p-1\}^B$
of cardinality $k+1$. Think of the $\alpha \in J$ as $B$-labelled
words over the alphabet $\{0,\ldots,p-1\}$.  Contracting $e_\alpha$
with $x_0+x_1+\ldots+x_{p-1}$ at position $i \in B$ yields $e_{\alpha'}$,
where $\alpha'$ is the word obtained from $\alpha$ by deleting the $i$-th
letter. By assumption, the resulting words $e_{\alpha'}, \alpha \in J$ are
linearly dependent, which means that at least two of them must coincide.

Summing up, $J$ consists of $k+1$ distinct words of length $|B|\geq k+1$
with the property that for each $i \in B$ the collection $J$ contains two
words that differ only at position $i$. By induction on $k$ we show that
this is impossible, i.e., that for $k+1$ distinct words of length $\geq
k+1$ over any alphabet there exist $k$ positions restricted to which all
words are distinct. For $k=0$ this is immediate: restricting a single
word to zero positions yields a single (empty) word. Assume that it is
true for $k-1$, and consider $k+1$ words of length $\geq k+1$. Set one
word $\alpha$ apart. Then there exist $k-1$ positions restricted to which
the remaining $k$ words are distinct. Restricted to those $k-1$ positions
$\alpha$ equals at most one word $\alpha'$ of the remaining words. So
by adding to the $k-1$ positions a position where $\alpha$ and $\alpha'$
differ we obtain $k$ positions restricted to which all words are distinct.

This contradiction shows that there exists an $i \in B$ and an $x \in
V$ such that the contraction of $U$ with $x$ at position $i$ still
has dimension $k+1$. As a consequence, contracting $\omega$ with $x$
at position $i$ yields a tensor outside $Y^{(k)}_{n-1}$, as claimed.
Thus $Y^{(k)}_\infty$ is defined by finitely many $G_\infty$-orbits of
equations. 

The variety $Y^{(k)}_\infty$ is defined by the vanishing
of $(k+1) \times (k+1)$-determinants. For what follows, it will be
convenient to understand these explicitly in terms of coordinates. 
The basis $x_0,\ldots,x_{p-1}$ gives rise to a basis $x_w, w \in
\{0,\ldots,p-1\}^{[n]}$ of $V^{\otimes [n]}$. The ring $R_n$ is the
polynomial ring in these variables. Under the embedding $\iota_n: R_n \to
R_{n+1}$ the variable $x_w$ is mapped to $x_{w0}$. Hence $R_\infty$ is
the polynomial ring in variables $x_w$ where $w$ runs over all infinite
words in $\{0,\ldots,p-1\}^\NN$ of finite support $\supp(w):=\{j
\in \NN \mid w_j \neq 0\}$; let us call these finitary words. In
these coordinates, a determinantal equation for $Y^{(k)}_\infty$ looks as
follows. Fix $k+1$ finitary words $w_i, i \in [k+1]$ and $k+1$ further
finitary words $w'_j,j \in [k+1]$ with the requirement that $\supp(w_i)
\cap \supp(w'_j)=\emptyset$ for all $i,j$. Then form the square matrix
\[ x[(w_i)_i, (w'_j)_j]:=(x_{w_i+w'_j})_{i,j \in [k+1]} \]
and its determinant 
\[ \Delta[(w_i)_i, (w'_j)_j]:=\det x[(w_i)_i, (w'_j)_j]. \]
All determinants defining $Y^{(k)}_\infty$ have this form.\footnote{The
convenient fact that the sum of two finitary words is again finitary
explains our choice of labelling $x_0,\ldots,x_{p-1}$.}

\subsection*{Noetherianity of flattening varieties}
Using this explicit understanding of the defining equations
for $Y^{(k)}_\infty$, we prove that $Y^{(k)}_\infty$, with its
Zariski-topology, is $G_\infty$-Noetherian. The proof is similar
to that in Chapter~\ref{ch:KFactor} for the Gaussian $k$-factor
model. In particular, we proceed by induction on $k$. For $k=0$
the variety $Y^{(0)}_\infty$ consists of a single point, namely,
$0$, and is certainly Noetherian. Now assume that $Y^{(k-1)}_\infty$
is $G_\infty$-Noetherian. By the discussion of flattening varieties,
$Y^{(k-1)}_\infty$ is defined by the orbits of finitely many
$k \times k$-determinants of flattenings, say $q$ of them. Let $\Delta_a,\
a \in [q]$ be those determinants. Then we may write
\begin{align*} 
Y^{(k)}_\infty & := Y^{(k-1)}_\infty \cup \bigcup_{a \in [q]} Z_a G_\infty, 
\text{ where}\\
Z_a &:= \{\omega \in Y^{(k)}_\infty \mid \Delta_a(\omega) \neq 0\}.
\end{align*}
As in the case of the $k$-factor model, it suffices to show that $Z_a$
is Noetherian under a suitable subgroup of $G_\infty$ stabilising it. To
this end, write $\Delta_a=\Delta[(w_i)_i, (w'_j)_j]$ for $k$ finitary
words $w_i$ and $k$ finitary words $w'_j$ with $\supp(w_i) \cap
\supp(w_j')=\emptyset$ for all $i,j$. Set $n:=\max \left( \bigcup_i
\supp(w_i) \cup \bigcup_j \supp(w'_j)\right)$ and observe that $\Delta_a$
is fixed by $H:=\{\pi \in \Sym(\infty) \mid
\pi|_{[n]}=1_{[n]}\}$, and
hence $Z_a$ is stabilised by $H$. We claim that $Z_a$ is $H$-Noetherian.
To prove this, let $J$ be the set of finitary words $w$ with $|\supp(w)
\setminus [n]| \leq 1$. In particular, all variables appearing in
$\Delta_a$ are in $J$. Let $Q$ be the open subset of $K^J$ where
$\Delta_a$ is non-zero. By Theorem~\ref{thm:kbyNmatrices} and the fact that adding
finitely many $H$-fixed variables to an $H$-Noetherian ring preserves
$H$-Noetherianity, the coordinate ring of $Q$ is $H$-Noetherian---here the
crucial point is that ``only one index runs off to infinity''.  We claim
that the projection $\pr:Z_a \to Q$ mapping a point to its coordinates
labelled by $J$ is an $H$-equivariant, closed embedding.
Equivariance is immediate. To see that $\pr$
is injective, we prove that on $Z_a$ any variable $x_w$
has an expression in terms of the variables labelled by $J$. We proceed by
induction on the cardinality of $\supp(w) \setminus [n]$. For cardinality
$0$ and $1$ the word $w$ lies in $J$ and we are done. So assume that the
cardinality is at least $2$ and that the statement is true for all
smaller cardinalities. Then we can split $w$ as $u+u'$ where 
$\supp(u) \cap \supp(u'),$ $\supp(u) \cap \supp(w'_j),$ and
$\supp(w_i) \cap
\supp(u')$ are all empty and where both $\supp(u)$ and $\supp(u')$
contain at least one element of $\NN \setminus [n]$. Then on $Z_a$ we
have 
\[ 0=\Delta[(w_1,\ldots,w_{k},u),(w'_1,\ldots,w'_{k},u')]
=\Delta_a \cdot x_{u+u'} - E  \]
where $E$ is an expression involving only variables whose supports
contain fewer elements of $\NN \setminus [n]$ than $\supp(w)$ does.  By the
induction hypothesis, these may be expressed in the variables labelled by
$J$, and as $\Delta_a$ is non-zero on $Z_a$, so can $x_{u+u'}=x_w$. To show
that $\pr$ is a closed embedding we note that the map $K[Q] \to K[Z]$ is
surjective: there is an expression for $E|_Z$ involving only $J$-labelled
variables, and dividing by $\Delta_a$ yields such an expression for
$x_w$. 

We conclude that $Z$ has the topology of a closed $H$-stable subspace
of $K^Q$ and is hence $H$-Noetherian. By Proposition~\ref{prop:ZG}
and the fact that finite unions of equivariantly Noetherian spaces
are equivariantly Noetherian, we find that $Y^{(k)}_\infty$ is
$G_\infty$-Noetherian. This concludes the proof of the theorem.

An important final remark is in order here: our proof of the theorem
shows that $X_\infty$ is defined by finitely many $G_\infty$-orbits of
equations, which is stronger than the theorem claims. In particular,
this stronger statement can be used to show that for each fixed
$\Delta$-variety there is a polynomial-time membership test. On the
other hand, it is typically {\em not} true that the ideal of $X_\infty$
is generated by finitely many $G_\infty$-orbits of polynomials; indeed,
this statement is already false for the cone over Segre. How to reconcile
this with the aforementioned conjecture \cite{Snowden10} that an ideal-theoretic
version of theorem should hold?  Well, by pulling back equations along
elements of $G_\infty$, we are implicitly pulling back equations along
tensor products of linear maps, and along permutations of tensor factors,
and along contractions, and along tensoring with $e_0$, but not along
flattening maps (though we did use, in the proof, that $\bX$ was closed
under flattening). This additional source of linear maps along which
to pull back equations may allow for an ideal-theoretic version of
the theorem.  For details see \cite{Snowden10,Draisma11d}.

\bibliographystyle{alpha}


\begin{thebibliography}{HMdC13}

\bibitem[AH07]{Aschenbrenner07}
Matthias Aschenbrenner and Christopher~J. Hillar.
\newblock Finite generation of symmetric ideals.
\newblock {\em Trans. Am. Math. Soc.}, 359(11):5171--5192, 2007.

\bibitem[AH08]{Aschenbrenner08}
Matthias Aschenbrenner and Christopher~J. Hillar.
\newblock An algorithm for finding symmetric {G}r{\"o}bner bases in infinite
  dimensional rings.
\newblock 2008.
\newblock Preprint available from \verb+http://arxiv.org/abs/0801.4439+.

\bibitem[AT03]{Aoki03}
S.~Aoki and A.~Takemura.
\newblock Minimal basis for connected {M}arkov chain over $3 \times 3 \times k$
  contingency tables with fixed two dimensional marginals.
\newblock {\em Australian and New Zealand Journal of Statistics}, 45:229--249,
  2003.

\bibitem[BB13]{Buczynska13}
Weronika Buczy\'nska and Jaros\l{}aw Buczy\'nski.
\newblock Secant varieties to high degree {V}eronese reembeddings,
  catalecticant matrices and smoothable {G}orenstein schemes.
\newblock {\em J. Alg. Geom.}, 2013.
\newblock To appear; preprint available from
  \verb+http://arxiv.org/abs/1012.3563+.

\bibitem[BD11]{Brouwer09e}
Andries~E. Brouwer and Jan Draisma.
\newblock Equivariant {G}r\"obner bases and the two-factor model.
\newblock {\em Math. Comput.}, 80:1123--1133, 2011.

\bibitem[Bor91]{Borel91}
Armand Borel.
\newblock {\em Linear Algebraic Groups}.
\newblock Springer-Verlag, New York, 1991.

\bibitem[CEF12]{Church12}
Thomas Church, Jordan~S. Ellenberg, and Benson Farb.
\newblock {FI}-modules: a new approach to stability for
  ${S}_n$-representations.
\newblock 2012.
\newblock Preprint, available from \verb+http://arxiv.org/abs/1204.4533+.

\bibitem[Coh67]{Cohen67}
Daniel~E. Cohen.
\newblock On the laws of a metabelian variety.
\newblock {\em J. Algebra}, 5:267--273, 1967.

\bibitem[Coh87]{Cohen87}
Daniel~E. Cohen.
\newblock Closure relations, {B}uchberger's algorithm, and polynomials in
  infinitely many variables.
\newblock In {\em Computation theory and logic}, volume 270 of {\em Lect. Notes
  Comput. Sci.}, pages 78--87, 1987.

\bibitem[DEKL13]{Draisma13a}
Jan Draisma, Rob~H. Eggermont, Robert Krone, and Anton Leykin.
\newblock Noetherianity for infinite-dimensional toric varieties.
\newblock 2013.
\newblock Preprint available from \verb+http://arxiv.org/abs/1306.0828+.

\bibitem[DK13]{Draisma11d}
Jan Draisma and Jochen Kuttler.
\newblock Bounded-rank tensors are defined in bounded degree.
\newblock {\em Duke Math. J.}, 2013.
\newblock To appear; preprint available from
  \verb+http://arxiv.org/abs/1103.5336+.

\bibitem[dLST95]{deLoera95}
Jes\'us~A. de~Loera, Bernd Sturmfels, and Rekha~R. Thomas.
\newblock Gr{\"o}bner bases and triangulations of the second hypersimplex.
\newblock {\em Combinatorica}, 15:409--424, 1995.

\bibitem[Dra10]{Draisma08b}
Jan Draisma.
\newblock Finiteness for the k-factor model and chirality varieties.
\newblock {\em Adv. Math.}, 223:243--256, 2010.

\bibitem[DS98]{Diaconis98}
Persi Diaconis and Bernd Sturmfels.
\newblock Algebraic algorithms for sampling from conditional distributions.
\newblock {\em Ann. Stat.}, 26(1):363--397, 1998.

\bibitem[DS06]{Drensky06}
Vesselin Drensky and Roberto~La Scala.
\newblock Gr{\"o}bner bases of ideals invariant under endomorphisms.
\newblock {\em J. Symb. Comput.}, 41(7):835--846, 2006.

\bibitem[DSS07]{Drton07}
Mathias Drton, Bernd Sturmfels, and Seth Sullivant.
\newblock Algebraic factor analysis: tetrads, pentads and beyond.
\newblock {\em Probab. Theory Relat. Fields}, 138(3--4):463--493, 2007.

\bibitem[DX10]{Drton08}
Mathias Drton and Han Xiao.
\newblock Finiteness of small factor analysis models.
\newblock {\em Annals of the Institute of Statistical Mathematics},
  62(4):775--783, 2010.

\bibitem[GSS05]{Garcia05}
Luis~D. Garcia, Michael Stillman, and Bernd Sturmfels.
\newblock Algebraic geometry of {B}ayesian networks.
\newblock {\em J. Symb. Comp.}, 39(3--4):331--355, 2005.

\bibitem[GW09]{Goodman09}
Roe Goodman and Nolan~R. Wallach.
\newblock {\em Symmetry, representations, and invariants}, volume 255 of {\em
  Graduate Texts in Mathematics}.
\newblock Springer, New York, 2009.

\bibitem[Hig52]{Higman52}
Graham Higman.
\newblock Ordering by divisibility in abstract algebras.
\newblock {\em Proc. Lond. Math. Soc., III. Ser.}, 2:326--336, 1952.

\bibitem[HMdC13]{Hillar13}
Christopher~J. Hillar and Abraham Mart{\'i}n~del Campo.
\newblock Finiteness theorems and algorithms for permutation invariant chains
  of {L}aurent lattice ideals.
\newblock {\em J. Symb. Comput.}, 50:314--334, 2013.

\bibitem[HS12]{Hillar09}
Christopher~J. Hillar and Seth Sullivant.
\newblock Finite {G}r\"obner bases in infinite dimensional polynomial rings and
  applications.
\newblock {\em Adv. Math.}, 221:1--25, 2012.

\bibitem[Kru60]{Kruskal60}
Joseph~B. Kruskal.
\newblock Well-quasi ordering, the tree theorem, and {V}azsonyi's conjecture.
\newblock {\em Trans. Am. Math. Soc.}, 95:210--225, 1960.

\bibitem[Lan65]{Lang65}
Serge Lang.
\newblock {\em Algebra}.
\newblock Addison-Wesley Publishing Company, Inc., Reading, Mass., 1965.

\bibitem[LM04]{Landsberg04}
Joseph~M. Landsberg and Laurent Manivel.
\newblock On the ideals of secant varieties of {S}egre varieties.
\newblock {\em Found. Comput. Math.}, 4(4):397--422, 2004.

\bibitem[LSL09]{LaScala09}
Roberto La~Scala and Viktor Levandovskyy.
\newblock Letterplace ideals and non-commutative {G}r\"obner bases.
\newblock {\em J. Symb. Comp.}, 44(10):1374--1393, 2009.

\bibitem[LW07]{Landsberg07}
J.~M. Landsberg and Jerzy Weyman.
\newblock On tangential varieties of rational homogeneous varieties.
\newblock {\em J. Lond. Math. Soc.(2)}, 76(2):513--530, 2007.

\bibitem[NW63]{NashWilliams63}
Crispin~St.J.A. Nash-Williams.
\newblock On well-quasi-ordering finite trees.
\newblock {\em Proc. Camb. Philos. Soc.}, 59:833--835, 1963.

\bibitem[OR11]{Oeding11}
Luke Oeding and Claudiu Raicu.
\newblock Tangential varieties of segre-veronese varieties.
\newblock 2011.
\newblock Preprint, avaibable from \verb+http://arxiv.org/abs/1111.6202+.

\bibitem[Rai12]{Raicu10}
Claudiu Raicu.
\newblock Secant varieties of {S}egre--{V}eronese varieties.
\newblock {\em Algebra \& Number Theory}, 6(8):1817--1868, 2012.

\bibitem[SHS07]{Hosten07}
Serkan Serkan~Ho\c{s}ten and Seth Sullivant.
\newblock A finiteness theorem for {M}arkov bases of hierarchical models.
\newblock {\em J. Comb. Theory, Ser. A}, 114(2):311--321, 2007.

\bibitem[Sno13]{Snowden10}
Andrew Snowden.
\newblock Syzygies of {S}egre embeddings and {$\Delta$}-modules.
\newblock {\em Duke Math. J.}, 2013.
\newblock To appear, preprint available from
  \verb+http://arxiv.org/abs/1006.5248+.

\bibitem[SS12a]{Sam12}
{Steven V. Sam} and Andrew Snowden.
\newblock {GL}-equivariant modules over polynomial rings in infinitely many
  variables.
\newblock 2012.
\newblock Preprint, available from \verb+http://arxiv.org/abs/1206.2233+.

\bibitem[SS12b]{Sam12b}
{Steven V. Sam} and Andrew Snowden.
\newblock Introduction to twisted commutative algebras.
\newblock 2012.
\newblock Preprint, available from \verb+http://arxiv.org/abs/1209.5122+.

\end{thebibliography}

\end{document}